\documentclass{amsart}
\usepackage{amsfonts, amssymb, latexsym, amscd}
\usepackage{amsmath, amsthm}
\usepackage[mathscr]{eucal}
\usepackage{enumerate}
\usepackage[pdftex]{graphicx, color}
\usepackage[usenames,dvipsnames]{xcolor}
\usepackage{dsfont}
\usepackage{pdfsync}
\usepackage{bbm}
\usepackage{esint}
\usepackage{mathtools}
\usepackage{verbatim}

\usepackage{color}

\DeclarePairedDelimiter{\ceil}{\lceil}{\rceil}

\theoremstyle{plain}
\newtheorem{thm}{Theorem}[section]
\newtheorem{cor}[thm]{Corollary}
\newtheorem{lem}[thm]{Lemma}
\newtheorem{prop}[thm]{Proposition}

\newtheoremstyle{named}{}{}{\itshape}{}{\bfseries}{.}{.5em}{\thmnote{#3's }#1}
\theoremstyle{named}

\newcounter{constnum}
\makeatletter
\def\const@nt#1{c_{#1}}
\def\c{\@ifnextchar[{\@with}{\@without}}
\def\@with[#1]{%
	\ifcsname constnum@#1\endcsname
	\else
		\stepcounter{constnum}%
		\expandafter\xdef\csname constnum@#1\endcsname{\theconstnum}%
	\fi
	\const@nt{\csname constnum@#1\endcsname}}
\def\@without{\stepcounter{constnum}c_{\theconstnum}}
\makeatother

\theoremstyle{definition}

\newtheorem*{claim}{Claim}

\newcommand{\pref}[1]{(\ref{#1})}

\newcommand{\NN}{\mathbb{N}}
\newcommand{\PP}{\mathbb{P}}
\newcommand{\ZZ}{\mathbb{Z}}
\newcommand{\RR}{\mathbb{R}}

\newcommand{\MM}{\mathcal{M}}
\newcommand{\EE}{\mathbb{E}}

\newcommand{\ve}{\varepsilon}

\newcommand{\vp}{\varphi}

\newcommand{\al}{\alpha}

\newcommand{\ga}{\gamma}

\newcommand{\ka}{\kappa}
\newcommand{\la}{\lambda}

\newcommand{\sig}{\sigma}

\newcommand{\om}{\omega}

\newcommand{\aint}{\fint}

\newcommand{\Ga}{\Gamma}
\newcommand{\De}{\Delta}

\newcommand{\La}{\Lambda}

\newcommand{\Om}{\Omega}
\newcommand{\norm}[1]{\|#1\|}
 
\newcommand{\rst}[1]{\ensuremath{{\mathbin\upharpoonright}%
\raise-.5ex\hbox{$#1$}}}

\newcounter{gscan}
\newcounter{btscan}
\newcounter{cscan}
\newcounter{hscan}
\newcounter{fscan}
\newcounter{pscan}
\newcounter{sscan}

\renewcommand{\thefscan}{F\arabic{fscan}}

\renewcommand{\tilde}{\widetilde}

\numberwithin{equation}{section}

\DeclareMathOperator*{\esssup}{ess\,sup}

\begin{document}
\title[Error Estimates for Parabolic Stochastic Homogenization]{Algebraic Error Estimates for the Stochastic Homogenization of Uniformly Parabolic Equations} 
\author{Jessica Lin}
\address{University of Wisconsin-Madison\\
Department of Mathematics\\
Madison, WI 53706}
\email[Jessica Lin]{jessica@math.wisc.edu}

\author{Charles K. Smart}
\address{Cornell University\\
Department of Mathematics\\
Ithaca, NY 14853}
\email[Charles K. Smart]{smart@math.cornell.edu}

\date{\today}

\begin{abstract}
This article establishes an algebraic error estimate for the stochastic homogenization of fully nonlinear uniformly parabolic equations in stationary ergodic spatio-temporal media. The approach is similar to that of Armstrong and Smart in the study of quantitative stochastic homogenization of uniformly elliptic equations. 
\end{abstract}

\maketitle

\section{Introduction}
We study quantitative stochastic homogenization of equations of the form
\begin{equation}\label{homeq}
\begin{cases}
u^{\ve}_{t}+F(D^{2}u^{\ve}, x/\ve, t/\ve^{2}, \om)=0 & \text{in}\quad U_{T},\\
u^{\ve}=g & \text{on}\quad\partial_{p} U_{T}, 
\end{cases}
\end{equation}
where $F$ is a random uniformly elliptic operator, determined by an element $\omega$ of some probability space, {$U_{T}:=U\times (0,T]\subsetneq \RR^{d+1}$ is a compact domain}, and $\partial_p U_T$ is the parabolic boundary.  In \cite{linhomog}, it was shown by one of the authors that under suitable hypotheses on the environment (namely stationarity and ergodicity of the operator in space and time), $u^{\ve}(\cdot, \cdot, \om)$ converges almost surely to a limiting function $u$ which solves
\begin{equation}\label{limeq}
\begin{cases}
u_{t}+\overline{F}(D^{2}u)=0 & \text{in}\quad U_{T},\\
u=g & \text{on}\quad\partial_{p} U_{T},
\end{cases}
\end{equation}
for a uniformly elliptic limiting operator $\overline{F}$ which is independent of $\omega$. Furthermore, a rate of convergence was established under additional quantitative ergodic assumptions. If the environment is strongly mixing with a prescribed logarithmic rate, then the convergence occurs in probability with a logarithmic rate, i.e.
{
\begin{equation}\label{genrate}
\PP\left[\sup_{U_{T}}|u^{\ve}(\cdot, \cdot, \om)-u(\cdot, \cdot)| \geq f(\ve) \right]\leq f(\ve), 
\end{equation}}
with ${f(\ve) \sim |\log \ve|^{-1}}$.  {In this article, we show that under the assumption of finite range of dependence, the homogenization occurs in probability with an algebraic rate, i.e. $f(\ve)\sim \ve^{\beta}$.}

\subsection{Background and Discussion}

For nondivergence form equations in the random setting, the pioneering works establishing the qualitative theory of homogenization (the convergence of $u^{\ve}\rightarrow u$) include (but are not limited to) the papers of Papanicolaou and Varadhan \cite{papvar1} and Yurinski{\u\i} \cite{yurestqual} for linear, nondivergence form uniformly elliptic equations, and Caffarelli, Souganidis, and Wang \cite{csw} for fully nonlinear uniformly elliptic equations. The study of quantitative stochastic homogenization seeks to establish error estimates for this convergence. For linear uniformly elliptic equations in nondivergence form, the first results were obtained by Yurinski{\u\i}  \cite{yurest2, yurest3}. Assuming that the environment satisfies an algebraic rate of decorrelation, his works present an algebraic rate of convergence for stochastic homogenization in dimensions $d\geq 5$. In dimensions $d=3,4$, the same result holds under the additional assumption of small ellipticity contrast, i.e. the ratio of ellipticities is close to 1. In dimension $d=2$, Yurinski{\u\i}'s results yield a logarithmic rate of convergence. 

 {For fully nonlinear equations, the first quantitative stochastic homogenization result appears in Caffarelli and Souganidis \cite{cs} for elliptic equations, and the parabolic case with spatio-temporal media was considered by one of the authors in \cite{linhomog}.  Both of these works obtain logarithmic convergence rates from logarithmic mixing conditions.  The approach of both papers is to adapt the obstacle problem method of Caffarelli, Souganidis, and Wang \cite{csw} to construct approximate correctors, which play the role of correctors in the random setting. }The logarithmic rate appears to be the optimal rate attainable with this approach.  This left open the question whether an algebraic rate similar to the results of Yurinski{\u\i} was attainable in the more general setting of fully nonlinear equations, and for problems in lower dimensions.

In the elliptic setting, this was addressed in \cite{asellip} by Armstrong and one of the authors.  They prove algebraic error estimates in all dimensions for the stochastic homogenization of fully nonlinear uniformly elliptic equations.  The main insight of their work was the introduction of a new subadditive quantity that (1) controls the solutions of the equation and (2) can be studied by adapting the regularity theory of Monge-Amp\`ere equations. Their method does not see the presence of correctors, and instead controls solutions indirectly via geometric quantities.

{The purpose of this article is to adapt the elliptic strategy to the parabolic spatio-temporal setting, which turns out to be subtle.  The approach of \cite{asellip} was to view the convex envelope of a supersolution as an approximate solution of the Monge-Amp\`ere equation
\begin{equation}\label{ellipma}
\det D^{2}w=1,
\end{equation}
for $w$ convex, and to then use ideas from the regularity theory of \pref{ellipma} (namely John's Lemma) to control the sublevel sets of $w$.  In the parabolic setting, we will show that the \textit{monotone envelope} of a supersolution of \pref{homeq} is an approximate solution of the analogous Monge-Amp\`ere equation 
\begin{equation}\label{genma}
-w_{t}\det D^{2}w=1,
\end{equation}
for $w$ parabolically convex (convex in space and non-increasing in time).  The equation \pref{genma} was first introduced by Krylov \cite{krylovconv}, and then it was further pointed out by Tso \cite{tso} that this was the most appropriate parabolic analogue of \pref{ellipma}. Regularity properties of  \pref{genma} have been studied by Guti{\'e}rrez and Huang in \cite{guthuang1, guthuang2}, and other parabolic Monge-Amp\`ere equations have been studied by Daskalopoulus and Savin in \cite{totiovidiuma}.  In spite of this work, the equation \eqref{genma} is still not as well-understood as \eqref{ellipma}.  In particular, there is no analogue of John's Lemma for sublevel sets of parabolically convex functions.  This forced us to develop an alternative approach (which can also be used in the elliptic setting) which replaces John's Lemma with a compactness argument.}

\subsection{Assumptions, and Statement of the Main Result}
We begin by stating the general assumptions on \pref{homeq}, and the precise statement of the main result. {We work in the stationary ergodic, spatio-temporal setting. We assume there exists an underlying probability space $(\Om, \mathcal{F}, \PP)$ such that

\begin{equation*}
\Om:=\left\{ F : \mathbb{S}^{d}\times\RR^{d+1}\rightarrow \RR\quad\text{satisfies}\quad \text{(F1)-(F4)}\right\}
\end{equation*}

\noindent where \pref{f1'}-\pref{f3'} will be specified below.  In particular, we have $F(X,y,s,\omega) = \omega(X,y,s)$. $\mathcal{F}$ is the Borel $\sigma$-algebra on $\Om$, and we assume that $\Om$ is equipped with a set of measurable measure-preserving transformations $\tau_ {(y', s')}: \Om\rightarrow \Om$ for each $(y',s')\in \RR^{d+1}$.} We also assume that $\partial_{p}U_{T}$ satisfies a uniform exterior cone condition, which allows us to construct global barriers (see \cite{parbar} for the precise assumption). Our hypotheses can be summarized as follows:

\begin{list}{ (\thefscan)}
{
\usecounter{fscan}
\setlength{\topsep}{1.5ex plus 0.2ex minus 0.2ex}
\setlength{\labelwidth}{1.2cm}
\setlength{\leftmargin}{1.5cm}
\setlength{\labelsep}{0.3cm}
\setlength{\rightmargin}{0.5cm}
\setlength{\parsep}{0.5ex plus 0.2ex minus 0.1ex}
\setlength{\itemsep}{0ex plus 0.2ex}
}

\item \label{f1'} \textit{Finite Range of Dependence}: 
For $A\subseteq\RR^{d+1}$, denote 
\begin{equation*}
\mathcal{B}(A):=\sigma\left\{F(\cdot, y, s, \om): (y,s)\in A\right\},
\end{equation*}
the $\sigma$-algebra generated by the operators $F$ defined on $A$. For $(x_{1},t_{1}),$
$(x_{2},t_{2})\in \RR^{d+1}$, let 
\begin{equation*}
d[(x_{1}, t_{1}), (x_{2}, t_{2})]:=(|x_{1}-x_{2}|^{2}+|t_{1}-t_{2}|)^{1/2}.
\end{equation*}
For $A,B\subseteq \RR^{d+1}$, let
\begin{equation}\label{metricdef}
d[A,B]:=\min \left\{d[(x,t),(y,s)]: (x,t)\in A, (y,s)\in B\right\}.
\end{equation}
{The finite range of dependence assumption is:
\begin{align}
\text{For all random variables}\,X&:\mathcal{B}(A)\rightarrow\RR, Y:\mathcal{B}(B)\rightarrow\RR\label{iid}\\
&\text{with $d[A,B]\geq 1$, $X, Y$ are $\PP$-independent.}\notag
\end{align}
}

\item\label{foops} \textit{Stationarity}:
For every $(M, \om)\in \mathbb{S}^{d}\times\Om$, where $\mathbb{S}^{d}$ denotes the space of $d\times d$ symmetric matrices with real entries, and for all $(y',s')\in\RR^{d+1}$, 
\begin{equation*}
F(M, y+y', s+s', \om)=F(M, y, s, \tau_{(y',s')}\om).
\end{equation*}
{ In fact, we only use this hypothesis for $(y',s') \in \ZZ^{d+1}$.}

\item \label{f2'} \textit{Uniform Ellipticity}. For a fixed choice of $\la, \La\in \RR$ with $0< \la\leq \La$, we define Pucci's extremal operators, 
\begin{align*}
&\MM^{+}(M)=\sup_{\la I \leq A \leq \La I}\left\{-tr(AM)\right\}=-\la \sum_{e_{i}>0} e_{i}-\La \sum_{e_{i}<0} e_{i},\\
&\MM^{-}(M)=\inf_{\la I \leq A \leq \La I}\left\{-tr(AM)\right\}=-\la \sum_{e_{i}<0} e_{i} -\La \sum_{e_{i}>0} e_{i}.
\end{align*}
We assume that $F(\cdot, y, s, \om)$ is uniformly elliptic for each $\om\in \Om$, i.e. for all $M,N\in \mathbb{S}^{d}$, and $(y,s, \om)\in \RR^{d+1}\times \Om$,  
\begin{equation*}
\MM^{-}(M-N)\leq F(M, y,s, \om)-F(N, y,s, \om)\leq \MM^{+}(M-N).
\end{equation*}
\item \label{f3'} \textit{Boundedness and Regularity of $F$}:  For every $R>0, \om\in \Om, M\in \mathbb{S}^{d}$ with $|M|\leq R$, 
\begin{equation*}
\left\{F(M, \cdot, \cdot, \om)\right\}\text{is uniformly bounded and uniformly equicontinuous on}~\RR^{d+1},
\end{equation*}
and there exists $K_{0}$ so that 
\begin{equation*}
\esssup_{\om\in \Om}\sup_{(y,s)\in \RR^{d+1}} |F(0, y, s, \om)|<K_{0}.
\end{equation*}

\noindent We also require that there exists a modulus of continuity $\rho[\cdot]$, and a constant $\sig>\frac{1}{2}$ such that for all $(M, y, s, \om)\in \mathbb{S}^{d}\times \RR^{d+1}\times \Om$, 
\begin{equation*}
|F(M,y_{1},s_{1},\om)-F(M, y_{2}, s_{2}, \om)|\leq \rho[(1+|M|)(|y_{1}-y_{2}|+|s_{1}-s_{2}|)^{\sig}]
\end{equation*}
where $|{\cdot}|$ denotes the standard Euclidean norm on $\RR^{d}$ and $\RR$ respectively. By applying \pref{f3'}, we have that
\begin{equation}\label{Mbnd}
\esssup_{\om\in\Om} \sup_{(y,s)\in \RR^{d+1}}|F(M, y, s, \om)|\leq C+\La |M|\leq C(1+|M|).
\end{equation}
\end{list}

Equipped with these assumptions, we now state the main result:

\begin{thm}\label{mainthm}
{
Assume \pref{f1'}-\pref{f3'}, and fix a domain $U_{T}$ and constant $M_{0}$.  There exists $C=C(\la, \La, d, M_{0})$ and a random variable $\mathcal{X}: \Om\rightarrow \RR$ with $\EE[\exp(\mathcal{X}(\om))]\leq C$, such that, whenever $u^{\ve}$ solves \pref{homeq}, $u$ solves \pref{limeq}, and
\begin{equation*}
1+K_{0}+\norm{g}_{C^{0,1}(\partial_{p}U_{T})}\leq M_{0},
\end{equation*}
then for any $p<d+2$, there exists a $\beta=\beta(\la, \La, d, p)>0$ such that 
\begin{equation}\label{serious}
\sup_{U_{T}} \left|u(x,t)-u^{\ve}(x, t, \om)\right|\leq C \left[1+\ve^{p}\mathcal{X}(\om)\right]\ve^{\beta}.
\end{equation}}
\end{thm}

{
The above theorem implies
\begin{equation}\label{finalrate}
\PP\left[\sup_{U_{T}} \left| u(x,t)-u^{\ve}(x, t, \om)\right|>C\ve^{\beta}\right]\leq C\exp(-\ve^{-p}),
\end{equation}
for $\beta > 0$ independent of the boundary data. It has recently been shown in the elliptic setting \cite{asdiv, scottjc, ottoco1, ottoco2} that quantitative estimates similar to \pref{serious} lead to a higher regularity theory at large scales. Although we do not discuss higher regularity results in this article, we are motivated by the recent progress in the elliptic setting to state our results in this form. }
\subsection{Notation and Conventions}
We mention some general notation and conventions used throughout the paper. The letters $\la, \La, K_{0}, T, U_{T}$ will always be used exclusively to refer to the constants stated in the assumptions. In the proofs, the letters $c,C$ will constantly be used as a generic constant which depends on these universal quantities, which may vary line by line, but is precisely specified when needed. We will always denote $\mathbb{S}^{d}$ as the set of symmetric $d\times d$ matrices with real entries, and $\mathbb{M}^{d}$ as the set of $d\times d$ matrices with real entries. We use the notation $|\cdot|$ to denote a norm on a finite-dimensional Euclidean space $(\RR, \RR^{d}, \RR^{d+1}$ or $\mathbb{S}^{d}$), or the Lebesgue measure on $\RR^{d+1}$, and we reserve $\norm{\cdot}$ to denote a norm on an infinite-dimensional function space. 

We choose to employ the parabolic metric 
\begin{equation*}
d[(x_{1}, t_{1}), (x_{2}, t_{2})]=(|x_{1}-x_{2}|^{2}+|t_{1}-t_{2}|)^{1/2}.
\end{equation*}
We point out that this equivalent to the metric 
\begin{equation*}
d_{\infty}[(x_{1}, t_{1}), (x_{2}, t_{2})]=\max\left\{|x_{1}-x_{2}|, |t_{1}-t_{2}|^{1/2}\right\}.
\end{equation*}
We say that $f\in C^{0, \al}$ if for any $(x,t), (y,s)\in \RR^{d+1}$, 
\begin{equation*}
|f(x,t)-f(y,s)|\leq \norm{f}_{C^{0, \al}}d[(x,t), (y,s)]^{\al}.
\end{equation*}

For sets, we use the notation $Q\subseteq \RR^{d+1}$ to represent an arbitrary space-time domain, i.e. $Q=Q'\times (t_{1}, t_{2}]$ where $Q'\subseteq \RR^{d}$. We define the parabolic boundary by
\begin{equation*} 
\partial_{p}Q:=(Q'\times \left\{t=t_{1}\right\})\cup (\partial Q'\times[t_{1}, t_{2})).
\end{equation*}
We use the convention that $\overline{Q}=Q\cup \partial_{p}Q$, and 
 \begin{equation*}
 Q(t):=\left\{x\in \RR^{d}: (x,t)\in Q\right\}.
 \end{equation*} 

We use the conventions
\begin{align*}
&B_{r}(\overline{x}, \overline{t})=B_{r}(\overline{x})\times \left\{t=\overline{t}\right\},\\
&\mathcal{B}_{r}(\overline{x}, \overline{t})=\left\{(x,t)\in \RR^{d+1}: d[(\overline{x}, \overline{t}), (x,t))]< r \right\},\\
&Q_{r}(\overline{x}, \overline{t})=B_{r}(\overline{x})\times (\overline{t}-r^{2}, \overline{t}].
\end{align*}
In general, $B_{r}, \mathcal{B}(r), Q_{r}$ are used to denote $B_{r}(0,0)$, $\mathcal{B}_{r}(0,0)$, and $Q_{r}(0,0)$ respectively. We point out that $\mathcal{B}_{r}$ and $Q_{r}$ are nothing more than the open balls generated by $d[\cdot, \cdot]$ and $d_{\infty}[\cdot, \cdot]$ respectively.

In addition to these sets, we work with a grid of parabolic cubes which partitions $\RR^{d+1}$. The grid boxes take the form 
\begin{align*}
G_{n}=\left[-\frac{1}{2}3^{n}, \frac{1}{2}3^{n}\right)^{d}\times (0, 3^{2n}].
\end{align*}
For every $(x,t)\in \RR^{d+1}$, we identify the cube
\begin{equation*}
G_{n}(x,t)=\left(3^{n}\lfloor 3^{-n}x+\frac{1}{2}\rfloor, 3^{2n}\lfloor 3^{-2n}t \rfloor\right)+ G_{n}.
\end{equation*}

\subsection{Outline of the Method and the Paper}

In Section \ref{sec:mu}, we define the appropriate parabolic analogue of the quantity introduced in \cite{asellip}. {We prove the basic properties of this quantity and describe how it controls solutions from one side.  In Section \ref{sec:pma}, we show how the quantity controls the behavior of solutions from the other side, utilizing the connection with the parabolic Monge-Amp\`ere equation.}  Here our primary innovation beyond \cite{asellip} appears.

In Section \ref{sec:qualstuff}, we construct the effective operator $\overline{F}$ using the asymptotic properties of our quantity, {and we also construct approximate correctors of \pref{homeq}.}
In Section \ref{sec:decaym}, we obtain a rate of decay on the second moments of this quantity, following closely the analysis of \cite{asellip}.   Finally, in Section \ref{sec:qptm}, we show how the rate on the second moments yields a rate of decay on $|u^{\ve}-u|$ in probability.

\section{A Subadditive Quantity Suitable for Parabolic Equations}\label{sec:mu}
\subsection{Defining $\mu(Q, \om, \ell, M)$}
We now define the quantity which will be used extensively throughout the rest of the paper. This quantity is a functional which measures the amount a function $u$ bends in space and time. We first recall some geometric objects relevant to the study of parabolic equations, and we refer the reader to \cite{krylovconv, wangreg1, cyriluis, guthuang2} for general references. We consider a subset $Q\subseteq\RR^{d+1}$, a fixed environment $\om\in \Om$, $\ell\in \RR$, and $M\in \mathbb{S}^{d}$. We then consider the set 
\begin{equation*}
S(Q, \om, \ell, M)=\left\{u\in C(Q): u_{t}+F(M+D^{2}u, x, t, \om)\geq \ell\quad\text{in}\quad Q\right\},
\end{equation*}
where the inequality is satisfied in the viscosity sense \cite{users}, and similarly, 
\begin{equation*}
S^{*}(Q, \om, \ell, M)=\left\{u\in C(Q): u_{t}+F(M+D^{2}u, x, t, \om)\leq \ell\quad\text{in}\quad Q\right\}.
\end{equation*}
To simplify the notation, we omit parameters when they are assumed to be 0, e.g. $S(Q, \om)$ refers to the choice $\ell=0$ and $M=0$. We say a function $u$ is parabolically convex if $u(\cdot, t)$ is convex for all $t$, and $u$ is non-increasing in $t$. For any function $u$, we define the monotone envelope to be the supremum of all parabolically convex functions lying below $u$. In particular, $\Ga^{u}$ has the following standard representation formula which can be taken as the definition:
\begin{equation*}
\Ga^{u}(x,t):=\sup \left\{ p\cdot x+h \mid p\cdot y +h\leq u(y,s), \forall (y,s)\in Q, s\leq t\right\}.
\end{equation*}
We point out that $\Ga^{u}$ depends on the domain $Q$, however we typically suppress this dependence.

At any point $(x_{0},t_{0})$, we compute the parabolic subdifferential,
\begin{align*}
\mathcal{P}((x_{0}, t_{0}); u):=\left\{(p,h)\subseteq\RR^{d+1}: \min_{x\in U, t\leq t_{0}} u(x,t)-p\cdot x=u(x_{0}, t_{0})-p\cdot x_{0}=h\right\},
\end{align*}
which may be empty.

We then say that for a domain $Q'\subseteq Q \subseteq \RR^{d+1}$, 

\begin{align*}
&\mathcal{P}(Q'; u):=\bigcup_{(x_{0},t_{0})\in Q'} \mathcal{P}((x_{0}, t_{0}); u)\\
&= \left\{(p,h): \exists\,(x_{0},t_{0})\in Q',\,s.t. \min_{(x,s)\in Q,\,s\leq t_{0}} u(x,s)-p\cdot x=u(x_{0}, t_{0})-p\cdot x_{0}=h\right\}.
\end{align*}

We now define the quantity
\begin{equation}\label{mudef}
\mu(Q, \om, \ell, M):=\frac{1}{|Q|}\sup\left\{ | \mathcal{P}(Q; \Ga^{u}) |: u\in S(Q, \om, \ell, M)\right\},
\end{equation}
where $|\cdot|$ denotes Lebesgue measure on $\RR^{d+1}$.

At this time, we also point out some properties of $\mu(Q, \om)$, which are critical for the analysis which follows. 
\begin{enumerate}
\item If $u$ is constant time, then $Q(t)$ is constant in time. The projection of $\mathcal{P}((x_{0},t); u)$ into $\RR^{d}$ is precisely the elliptic subdifferential of the convex envelope of $u$. We denote the elliptic subdifferential by $\partial \Ga^{u}[t](\cdot; \cdot)$. This shows that after an appropriate projection and renormalization, $\mu$ as defined in \pref{mudef} reduces to the quantity defined in \cite{asellip}. 
\item This quantity respects the scaling on domains with parabolic scaling. For each $u\in S(G_{n}, \om)$, let $u_{n}(x,t):=3^{-2n}u(3^{n}x, 3^{2n}t)\in S(G_{0}, \om)$. Under this scaling, if $(p,h)\in \mathcal{P}(G_{n}; u)$, then $(3^{-n}p, 3^{-2n}h)\in \mathcal{P}(G_{0}; u_{n})$. Thus, we have that

\begin{align*}
|\mathcal{P}(G_{n}; u)|=3^{n(d+2)}|\mathcal{P}(G_{0}; u_{n})|.
\end{align*}
This shows us that in order to prove statements for $\mu(G_{n}, \om)$, it is enough to prove statements for $\mu(G_{0}, \om)$, and rescale.

\item If $w\in C^{2}(Q)$ and parabolically convex, then $\mathcal{P}((x_{0}, t_{0}); w)$ reduces to 
\begin{equation*}
\mathcal{P}((x,t); w)=(Dw(x,t), w(x,t)-Dw(x,t)\cdot x).
\end{equation*}
If we interpret $\mathcal{P}((\cdot, \cdot); w)=\mathcal{P}[w](\cdot, \cdot): \RR^{d+1}\rightarrow \RR^{d+1}$, then by a standard computation,
\begin{equation*}
\det \mathcal{D}\mathcal{P}[w]=-w_{t} \det D^{2}w, 
\end{equation*}
where $\mathcal{D}\mathcal{P}[w]=D_{t,x}\mathcal{P}[w]$. We point out that the right hand side is precisely the Monge-Amp\`ere operator first introduced in \cite{krylovconv, tso}. Therefore, by applying the area formula \cite{evansgarbook}, 

\begin{equation*}
\frac{1}{|Q|}| \mathcal{P}(Q; w) |=\frac{1}{|Q|} \int_{Q} \det \mathcal{D}\mathcal{P}[w] ~dx dt=\frac{1}{|Q|} \int_{Q} -w_{t} \det D^{2}w~dxdt.
\end{equation*}

This shows the formal connection between the quantity $\frac{|\mathcal{P}(Q; \Ga^{u})|}{|Q|}$ and the parabolic Monge-Amp\`ere equation. We will explore this connection further in Section \ref{sec:pma}. 

\end{enumerate}

As introduced in \cite{asellip}, we now define $\mu^{*}(G_{n}, \om)$, which will serve as the analogous quantity corresponding to subsolutions. We define the involution operator $\pi(\om)=\om^{*}$ by 
\begin{equation*}
F(M, x, t, \om^{*}):=-F(-M, x, t, \om)\quad\text{for}\quad (M, x, t,\om)\in \mathbb{S}^{d}\times \RR^{d+1}\times \Om.
\end{equation*}
(Recall we assumed $\Omega$ is the space of operators $F$.)  We point out that $\pi:\Om\rightarrow \Om$ is a bijection, and $\om^{**}=\om$. Moreover, for $u\in C(\overline{Q})$, 
\begin{equation*}
u_{t}+F(-M+D^{2}u, x, t, \om^{*})\geq -\ell\quad \Longleftrightarrow v:=-u~\text{solves}~v_{t}+F(M+D^{2}v, x, t, \om)\leq \ell,
\end{equation*}
\noindent in the viscosity sense. Therefore, we define

\begin{align}\label{mu*def}
\mu^{*}(Q, \om, \ell, M)&:=\frac{1}{|Q|}\sup\left\{ | \mathcal{P}(Q; \Ga^{u}) |: u\in S(Q, \om^{*}, -\ell, -M)\right\}\\
&=\mu(Q, \om^{*}, -\ell, -M)\notag\\
&=\frac{1}{|Q|}\sup\left\{ | \mathcal{P}(Q; \Ga^{-u}) |: u\in S^{*}(Q, \om, \ell, M)\right\}\notag.
\end{align}

Since $\pi(\om)=\om^{*}$ is an $\mathcal{F}$-measurable function on $\Om$, we define the pushforward 
\begin{equation*}
\pi_{\#}\PP(E):=\PP[\pi^{-1}(E)].
\end{equation*}
This justifies that $\mu^{*}(Q, \om)$ enjoys the analogous properties of $\mu(Q,\om)$ for subsolutions. Throughout the paper, we will focus on showing results for $\mu(Q, \om)$ and the analogous statements hold for $\mu^{*}(Q, \om)$. 

\subsection{Regularity Properties of $\mu(Q, \om)$}

First, we show that $\mu(Q, \om)$ controls the behavior of supersolutions on the parabolic boundary from one side. 
\begin{lem}\label{muabp}
There exists a constant $\c[ptf]=\c[ptf](d)$ such that for every $\om\in \Om$, $(x,t)\in \RR^{d+1}$, $n\in \ZZ$, $u\in S(G_{n}(x,t), \om)$, 
\begin{equation}\label{muabpineq}
\inf_{\partial_{p}G_{n}(x, t)} u\leq \inf_{G_{n}(x, t)} u+\c[ptf]3^{2n}\mu(G_{n}(x,t), \om)^{1/(d+1)}.
\end{equation}
\end{lem}

\begin{proof}
Without loss of generality, in light of the scaling of $\mu(\cdot, \om)$, it is enough to prove the statement for $G_{0}$. Moreover, we assume that $a:=\inf_{\partial_{p}G_{0}} u-\inf_{G_{0}} u>0$. Let $(x_{0}, t_{0})\in G_{0}$ such that $u(x_{0}, t_{0})=\inf_{G_{0}} u$. This implies that for all $|p|\leq \frac{1}{\sqrt{d}}a$, for all $(y,s)\in \partial_{p}G_{0}$, 

\begin{align*}
u(x_{0}, t_{0})-p\cdot x_{0}=\inf_{\partial_{p}G_{0}} u-a-p\cdot x_{0}&\leq u(y,s)-p\cdot y+p\cdot (y-x_{0})-a\\
&\leq u(y,s)-p\cdot y+a-a= u(y,s)-p\cdot y,
\end{align*}
since $|y-x_{0}|\leq \sqrt{d}$. This implies that the minimum of the map $(x,t)\rightarrow u(x,t)-p\cdot x$ occurs in the interior of $G_{0}$. Thus, for all $|p|\leq \frac{1}{\sqrt{d}}a$, there exists a choice of $h$ such that $(p,h)\in \mathcal{P}(G_{0}; u)$.

For each fixed $p$, with $|p|\leq \frac{1}{\sqrt{d}}a$, we examine which values of $h$ are included in $\mathcal{P}(G_{0}; u)$. Recall that 
\begin{equation*}
h=h(t_{0})=\min_{(x,t)\in G_{0}, t\leq t_{0}} u(x,t)-p\cdot x.
\end{equation*}
 In particular, for each fixed $p$, $h(\cdot):\RR\rightarrow\RR$ is continuous. Therefore, this implies that $(p,h)\in \mathcal{P}(G_{0}; u)$ for all $h\in [u(x_{0}, t_{0})-p\cdot x_{0}, \inf_{\partial_{p}G_{0}} (u(x,t)-p\cdot x)]$. 

Combining these observations, this yields that
\begin{equation}\label{cont}
\left\{(p,h): |p|\leq \frac{1}{\sqrt{d}}a,\, \inf_{G_{0}} u-p\cdot x_{0}\leq h\leq \inf_{\partial_{p}G_{0}} u-p\cdot x \right\}\subseteq \mathcal{P}(G_{0}; u).
\end{equation}
The left side of \pref{cont} contains a hypercone in $\RR^{d+1}$ with base radius $\frac{1}{\sqrt{d}}a$, and height $a$. 

Therefore, we have that for $c=c(d)$,  
\begin{equation*}
ca^{d+1}\leq \left| \mathcal{P}(G_{0}; u)\right|. 
\end{equation*}

Since $\mathcal{P}(G_{0}; u)\subseteq \mathcal{P}(G_{0};\Ga^{u})$, this yields
\begin{equation*}
a\leq \left(\frac{1}{c}\right)^{1/(d+1)}\left(\frac{|\mathcal{P}(G_{0}; \Ga^{u})|}{|G_{0}|}\right)^{1/(d+1)}\leq \c[ptf]\mu(G_{0}, \om)^{1/(d+1)}
\end{equation*}
with $\c[ptf]=\c[ptf](d)$. 
\end{proof}

We now recall several results regarding the regularity of $\Ga^{u}$. These results and their proofs can be found in \cite{krylovconv, tso, wangreg1, cyriluis}. 

It is sometimes useful to use an alternative representation formula for the monotone envelope, in terms of its contact points. We state the lemma here and refer the reader to \cite{cyriluis} for the proof. 
\begin{lem}[\cite{cyriluis}, Lemma 4.5]\label{repform}
$\Ga^{u}$ satisfies the following alternative representation formula:
\begin{align*}
\Ga^{u}(x,t)=\inf \left\{ \sum_{i=1}^{d+1}\la_{i}u(x_{i}, t_{i}): \sum_{i=1}^{d+1} \la_{i} x_{i}=x, t_{i}\in [0,t], \sum_{i=1}^{d+1} \la_{i}=1, \la_{i}\in [0,1]\right\}.
\end{align*}

In particular, if 
\begin{equation*}
\Ga^{u}(x^{0}, t^{0})=\sum_{i=1}^{d+1}\la_{i} u(x_{i}^{0}, t^{0}_{i})\quad\text{with}\quad \la_{i}>0,
\end{equation*}
then 
\begin{itemize}
\item for all $i=1, \ldots, d+1,$ $\Ga^{u}(x^{0}_{i}, t^{0}_{i})=u(x^{0}_{i}, t^{0}_{i}).$\\
\item $\Ga^{u}$ is constant with respect to $t$ and linear with respect to $x$ in the convex set co$\left\{(x^{0}_{i}, t^{0}), (x^{0}_{i}, t^{0}_{i})\right\}_{i=1}^{d+1}$, the convex hull of $\left\{(x^{0}_{i}, t^{0}), (x^{0}_{i}, t^{0}_{i})\right\}_{i=1}^{d+1}$.
\end{itemize}
\end{lem}

As a consequence of this representation formula, it is natural to expect that $\Ga^{u}$ inherits regularity properties of the function $u$. 

\begin{lem}[\cite{cyriluis}, Lemma 4.11]\label{mereg}
Suppose that $u_{t}+\MM^{+}(D^{2}u)\geq -1$. The function $\Ga^{u}$ is $C^{1,1}$ with respect to $x$ and Lipschitz continuous with respect to $t$. In particular, $\mathcal{P}[\Ga^{u}]: \RR^{d+1}\rightarrow \RR^{d+1}$ is Lipschitz continuous with respect to $(x,t)$. 
\end{lem}

In addition, if $u$ is a supersolution to Pucci's equation, it turns out that $\Ga^{u}$ is actually a supersolution to a linear equation almost everywhere:

\begin{lem}[\cite{cyriluis}, Lemma 4.12]\label{meeq}
Suppose that $u_{t}+\MM^{+}(D^{2}u)\geq -1$. The partial derivatives $(\Ga^{u}_{t}, D^{2}\Ga^{u})$ satisfy almost everywhere, 
\begin{equation*}
\Ga^{u}_{t}-\la \Delta \Ga^{u}\geq -1\quad\text{in}\quad Q\cap \left\{u=\Ga^{u}\right\}.
\end{equation*}
\end{lem}

We next establish a lemma which shows that  in fact, $|\mathcal{P}(Q; u)|=|\mathcal{P}(Q; \Ga^{u})|.$ As previously mentioned, it is immediate that $\mathcal{P}(Q; u)\subseteq \mathcal{P}(Q; \Ga^{u})$, and thus $|\mathcal{P}(Q; u)|\leq |\mathcal{P}(Q; \Ga^{u})|$. In order to conclude, it is enough to show the following lemma, which is the parabolic analogue of Lemma 2.4 of \cite{asellip}. 

\begin{lem}\label{cs}
Let $Q\subseteq \RR^{d+1}$ denote an open subset, with $u\in C(Q)$, $(x_{0}, t_{0})\in Q$, and $r>0$ such that 
\begin{equation*}
Q_{r}(x_{0}, t_{0})\subseteq \left\{(x,t)\in Q: \Ga^{u}(x,t)<u(x,t)\right\}=\left\{\Ga^{u}<u\right\}.
\end{equation*}
Then $|\mathcal{P}(Q_{r}(x_{0}, t_{0}); \Ga^{u})|=0.$
\end{lem}

\begin{proof}
Without loss of generality, we may assume that $r<1$. Moreover, by a covering argument, it is enough to show that $|\mathcal{P}(Q_{r}(x_{0}, t_{0}); \Ga^{u})|=0$ assuming that $Q_{3r}(x_{0}, t_{0})\subseteq \left\{\Ga^{u}<u\right\}$. 

Suppose for the purposes of contradiction that $|\mathcal{P}(Q_{r}(x_{0}, t_{0}); \Ga^{u})|>0$. Since the measure is positive, by the Lebesgue density theorem, almost every $(p,h)\in \mathcal{P}(Q_{r}(x_{0}, t_{0}); \Ga^{u})$ is a density point. We mention that the density theorem still holds for parabolic cylinders and we refer the reader to the appendix of \cite{cyriluis} for a proof. We next have the following claim:
\begin{claim} There exists $(x',t')\in Q_{r}(x_{0}, t_{0})$ and $(p,h)\in \mathcal{P}((x',t'); \Ga^{u})$ so that $(p,h)$ is a Lebesgue density point of $\mathcal{P}(Q_{r}(x_{0}, t_{0}); \Ga^{u})$, and also, $p\in \partial \Ga^{u}[t'](x')$ is a Lebesgue density point of $\partial \Ga^{u}[t'](B_{r}(x_{0}))$. 
\end{claim}
\noindent This follows from applying the Lebesgue density theorem to both $\mathcal{P}(Q_{r}(x_{0}, t_{0}); \Ga^{u})$ and $\partial \Ga^{u}[t'](B_{r}(x_{0}))$ for some $t'$ where $|\partial \Ga^{u}[t'](B_{r}(x_{0}))|>0$. By adding an affine function in space and translating, we may assume that $x_{0}=0$, $t_{0}=0$, $\Ga^{u}(x',t')=0$, and $(p',h')=(0,0)$.  

Since $0$ is a Lebesgue density point of $\partial \Ga^{u}[t'](B_{r})$, for any $\overline{x}\in \partial B_{r}$ for $r$ sufficiently small, there exists a $\overline{p}\in \partial \Ga^{u}[t'](B_{r})\setminus 0$ such that 
\begin{equation*}
\overline{p}\cdot \overline{x}\geq \frac{3}{4}|\overline{p}||\overline{x}|.
\end{equation*}
Suppose that $\overline{p}\in \partial \Ga^{u}[t'](y)$. Since $\Ga^{u}(\cdot, t')\geq 0$ in $B_{r}$, this implies that for any $\al\geq 2$, 
\begin{equation*}
\Ga^{u}(\al \overline{x}, t')\geq \Ga^{u}(y, t')+\overline{p}\cdot (\al \overline{x}-y) \geq \al \overline{p}\cdot\overline{x}-\overline{p}\cdot y\geq \frac{3}{4} \al r|\overline{p}|-r|\overline{p}|>0.
\end{equation*}

This and the monotonicity of $\Ga^{u}$ allows us to conclude that 
 \begin{equation*}
 \Ga^{u}>0 \quad\text{on}\quad \left\{|x|\geq 2r, \forall t\leq t'\right\}.
 \end{equation*}
 
 Moreover, we point out that since $(0,0)$ is a Lebesgue point of $\mathcal{P}(Q_{r}; \Ga^{u})$, for each $|x|\leq r<1$, there exists $(p_{2}, h_{2})\in \mathcal{P}(Q_{r}; \Ga^{u})\setminus (0,0)$ 
 
 \begin{equation*}
p_{2}\cdot x+h_{2} r^{2}>\frac{3}{4}|(p_{2}, h_{2})||(x, r^{2})|>0.
 \end{equation*}
 
 Let $(p_{2}, h_{2})\in \mathcal{P}((y,s); \Ga^{u})$ for $(y,s)\in Q_{r}$. This implies that for all $t\leq s$, for all $|x|\leq r$, since $h_{2}\geq 0$ and $r<1$, 
 
 \begin{equation*}
 \Ga^{u}(x,t)\geq p_{2}\cdot x+h_{2}=p_{2}\cdot x+h_{2}r^{2}+h_{2}(1-r^{2})>0.
 \end{equation*}
 \medskip
 \noindent Therefore, for all $t\leq -r^{2}$, we conclude again that $\Ga^{u}>0$. This implies that 
 \begin{equation*}
\Ga^{u}>0\quad\text{in}\quad (Q\setminus Q_{2r})\cap \left\{t\leq t'\right\}.
 \end{equation*}
 
However, since $u>\Ga^{u}$ on $Q_{3r}$, this implies that $u>0$ on all of $Q\cap\left\{t\leq t'\right\}$. This contradicts that $\Ga^{u}(x', t')=0$, and hence we have the claim. 
\end{proof}

This regularity allows us to establish

\begin{lem}\label{area}
Assume that $Q\subseteq \RR^{d+1}$ is bounded and open, and $u\in C(Q)$ satisfies
\begin{equation*}
u_{t}+\MM^{+}(D^{2}u)\geq -1,
\end{equation*}
then there exists $\c[contact]=\c[contact](\la, d)$ such that
\begin{equation}\label{contact}
\left| \mathcal{P}(Q; \Ga^{u})\right| \leq \c[contact] \left| \left\{u=\Ga^{u}\right\}\cap Q\right|.
\end{equation}
\end{lem}

\begin{proof}
Given the regularity of $\Ga^{u}$ established by Lemma \ref{mereg}, we apply the area formula for Lipschitz functions to conclude that 
\begin{align*}
\left| \mathcal{P}(Q; \Ga^{u})\right| &=\int _{Q} \det \mathcal{D}\mathcal{P}(\Ga^{u})= \int_{Q\cap \left\{u=\Ga^{u}\right\}} -\Ga^{u}_{t}\det D^{2}\Ga^{u}\\
&=\la^{-d} \int_{Q\cap\left\{u=\Ga^{u}\right\}} -\Ga^{u}_{t}\det D^{2}\la \Ga^{u}.
\end{align*}

By applying the geometric-arithmetic mean inequality and Lemma \ref{meeq}, we have that 
\begin{align*}
\la^{-d} \int_{Q\cap\left\{u=\Ga^{u}\right\}} -\Ga^{u}_{t}\det D^{2}\la \Ga^{u}\,dxdt&\leq c(\la, d)\int_{Q\cap\left\{u=\Ga^{u}\right\}} \left[ -\Ga^{u}_{t}+\la \Delta \Ga^{u}\right]^{d+1}\, dxdt\\
&\leq c\int _{Q\cap \left\{u=\Ga^{u}\right\}}1\,dxdt=c\left|\left\{u=\Ga^{u}\right\}\cap Q\right|,
\end{align*}
which yields \pref{contact}. 
\end{proof}
{
We next claim that $\lim_{n\rightarrow\infty} \mu(G_{n}, \om)$ exists almost surely. This will follow by an application of the subadditive ergodic theorem of Akcoglu and Krengel \cite{akerg} to the quantity 
\begin{equation*}
\sup_{u\in S(G_{n}, \om)}|\mathcal{P}(G_{n}; \Ga^{u})|.
\end{equation*}
We point out that the result of \cite{akerg} also holds for cubes with parabolic scaling. In order to verify the hypotheses, we first show a decomposition property of $\mu(\cdot, \om)$:}

\begin{lem}\label{lemdecomp}
For each $\om\in \Om, n\in \ZZ, m\in \NN$, 
\begin{equation}\label{decomp}
\mu(G_{n+m}, \om)\leq \aint_{G_{n+m}}\mu(G_{n}(x,t), \om)~dxdt.
\end{equation}
\end{lem}

\begin{proof}
Let $u\in S(G_{n+m}, \om)$. By applying Lemma \ref{area}, we have that for each $(x,t)\in G_{n+m}$, 
\begin{equation*}
|\mathcal{P}(G_{n+m}\cap \partial_{p}G_{n}(x,t); \Ga^{u})|=0.
\end{equation*}
Therefore, 
\begin{align*}
|\mathcal{P}(G_{n+m}; \Ga^{u})|\leq\sum_{\left\{G=G_{n}(x,t)\subseteq G_{n+m}\right\}} |\mathcal{P}(G; \Ga^{u})|&=\int_{G_{n+m}}\frac{|\mathcal{P}(G_{n}(x,t); \Ga^{u})|}{|G_{n}|}~dxdt\\
&\leq \int_{G_{n+m}}\frac{|\mathcal{P}(G_{n}(x,t); \Ga^{\tilde{u}})|}{|G_{n}|}~dxdt
\end{align*}
where $\tilde{u}=u\rst{G_{n}(x,t)}$, for $(x,t)\in G_{n+m}$. By taking supremum of both sides, we have \pref{decomp}.
\end{proof}

Lemma \ref{lemdecomp} shows that $\EE[\mu(G_{n}, \om)]$ is non-increasing in $n$. We next show universal bounds for $\mu$. 

\begin{lem}\label{bnds}
There exists $\c[lower]=\c[lower](\la, \La, d)>0$ and $\c[upper]=\c[upper](\la, \La, d)>0$ so that for every $\om\in \Om$, $n\in \ZZ$, for every $M\in \mathbb{S}^{d}$, for every $\ell\in \RR$, 
\begin{equation}\label{bndest'}
\c[lower] \inf_{(x,t)\in G_{n}} (F(M, x,t, \om)-\ell)^{d+1}_{+}\leq \mu(G_{n}, \om, \ell, M)\leq \c[upper] \sup_{(x,t)\in G_{n}} (F(M, x,t, \om)-\ell)_{+}^{d+1}.
\end{equation}
\end{lem}

\begin{proof}
We fix $M\in \mathbb{S}^{d}$, and without loss of generality, we assume that $\ell=0$. By Lemma \ref{area}, the right inequality holds by scaling and rearranging. To prove the left inequality, we note that letting 
\begin{align*}
&\eta:=\inf_{(x,t)\in G_{n}} (F(M, x,t, \om))_{+}\\
&\vp(x,t):=-\frac{\eta}{4}t+\frac{\eta}{4d\La}|x|^{2},
\end{align*}
for each $(x,t)\in G_{n}$, 
\begin{align*}
\vp_{t}+F(M+D^{2}\vp, x, t, \om)&\geq \vp_{t}+\MM^{-}(D^{2}\vp)+F(M, x,t, \om)\\
&=-\frac{\eta}{4}-\frac{\eta}{2}+F(M, x,t, \om)\geq 0.
\end{align*}
Therefore, $\vp\in S(G_{n}, \om, M)$, and hence 
\begin{equation*}
\mu(G_{n}, \om, M)\geq \frac{|\mathcal{P}(G_{n}; \vp)|}{|G_{n}|}=\frac{1}{|G_{n}|}\int -\vp_{t} \det D^{2}\vp=\c[lower] \eta^{d+1}.
\end{equation*}
\end{proof}
In particular, we mention that \pref{bndest} implies
\begin{equation}\label{bndest}
\c[lower] \inf_{(x,t)\in G_{n}} (F(M, x,t, \om)-\ell)^{d+1}_{+}\leq \mu(G_{n}, \om, \ell, M)\leq \c[upper] [K_{0}(1+|M|)-\ell]_{+}^{d+1}.
\end{equation}

Using the previous two lemmas, we establish

\begin{cor}\label{setok}$\displaystyle \lim_{n\rightarrow\infty} \mu(G_{n}, \om)$ exists almost surely. 
\end{cor}
{
\begin{proof}
We apply the subadditive ergodic theorem to the quantity 
\begin{equation*}
R(G_{n}, \om):=\sup_{u\in S(G_{n}, \om)}|\mathcal{P}(G_{n}; \Ga^{u})|.
\end{equation*}
We note by the stationarity of $F(\cdot, \cdot, \cdot, \om)$, it follows that $R(\cdot, \om)$ is stationary. By Lemma \ref{lemdecomp}, Lemma \ref{bnds}, and \pref{f3'}, $R(\cdot, \om)$ is subadditive on parabolic cubes and bounded almost surely. An application of the subadditive ergodic theorem yields the claim. 
\end{proof}
In light of the ergodicity assumption, the limit is a constant almost surely. We note that if $\lim_{n\rightarrow\infty} \mu(G_{n}(x,t), \om)=0$, then by \pref{muabpineq}, we obtain a type of comparison principle in the limit. In the next section, we will show that if the limit is strictly positive, then we obtain control of the growth of an optimizing supersolution. }

\section{Strict Convexity of Quasi-Maximizers}\label{sec:pma}
The results in this section are completely deterministic, and we suppress all dependencies on the random parameter $\om$. We show that $|\mathcal{P}(Q; \Ga^{u})|$ yields geometric information about the function $u\in S(Q)$. More specifically, for some $n\leq 0,$ if $\frac{|\mathcal{P}(G_{n}(x,t); \Ga^{u})|}{|G_{n}|}\approx 1$ for all $(x,t)\in G_{0}$, then the optimizing supersolution for $\mu(G_{0})$ is strictly convex. In particular, up to an affine transformation, the optimizing supersolution bends upwards on $\partial_{p}G_{0}$. 

Formally, if $\vp$ is parabolically convex with classical derivatives, then for $n$ sufficiently small, by the Lebesgue differentiation theorem,
\begin{equation*}
-\vp_{t}(x,t)\det D^{2}\vp(x,t)\approx \aint_{G_{n}(x,t)} -\vp_{s} \det D^{2}\vp\, dyds=\frac{|\mathcal{P}(G_{n}(x,t); \vp)|}{|G_{n}|}.
\end{equation*}
Therefore, if $\frac{|\mathcal{P}(G_{n}(x,t); \vp)|}{|G_{n}|}\approx 1$ for all $(x,t)$, this is related to solving the parabolic Monge-Amp\`ere equation $-\vp_{t}\det D^{2}\vp=1$. This idea originated in \cite{asellip}, where given an equivalent measure condition for the elliptic subdifferential of the convex envelope, the authors conclude that the optimizing supersolution is strictly convex.

In this article, we first utilize the regularity properties of $u\in S(G_{0})$ to show that the time derivatives and Hessian of $w=\Ga^{u}$ are uniformly bounded above almost everywhere. In particular, this bound only depends on the ellipticity constants and dimension. Using the structure of \pref{genma}, we then obtain that the time derivative and Hessian are also strictly positive almost everywhere, which allows us to conclude that the solution must be strictly convex. We mention that this approach can also be applied to the elliptic setting of \cite{asellip} to produce an alternative argument. 

We first show that by using that $u\in S(G_{0})$, the monotone envelope $\Ga^{u}$ {satisfies} a uniform upper bound on the time derivative and Hessian at its contact points. Recall that by Lemma \ref{mereg}, $\Ga^{u}$ is Lipschitz continuous in time and $C^{1,1}$ in space. Therefore, we may represent $(p,h)\in\mathcal{P}((x_{0}, t_{0}); \Ga^{u})$ by $(D\Ga^{u}(x_{0}, t_{0}), u(x_{0}, t_{0})-D\Ga^{u}(x_{0}, t_{0})\cdot x_{0})\in\mathcal{P}((x_{0}, t_{0}); \Ga^{u})$. 

\begin{lem}\label{lem:upbnds}
Let $u\in S(G_{0})$, and suppose 
\begin{equation}\label{mu_up}
\frac{|\mathcal{P}(G_{-2}(x,t); \Ga^{u})|}{|G_{-2}|}\leq 2\quad\text{for all}\quad (x,t)\in G_{0}.
\end{equation}
There exists $\ga=\ga(\la, \La, d)$, such that for all $(x_{0}, t_{0})\in Q_{1/4}(0,1)\cap\left\{u=\Ga^{u}\right\}$, we have that for all $(y,s)\in Q_{1/4}(x_{0}, t_{0})$,
\begin{equation}\label{upest}
\Ga^{u}\left(y, s\right)\leq \Ga^{u}(x_{0}, t_{0})+D\Ga^{u}(x_{0}, t_{0})\cdot (y-x_{0})+\ga.
\end{equation}
\end{lem}

\begin{proof}
We point out that by the monotonicity of $\Ga^{u}$, it is enough if we can show that for all $y\in B_{1/4}(x_{0})$ where $u(x_{0}, t_{0})=\Ga^{u}(x_{0}, t_{0})$, 
\begin{equation}
\Ga^{u}\left(y, t_{0}-\frac{1}{16}\right)\leq \Ga^{u}(x_{0}, t_{0})+D\Ga^{u}(x_{0}, t_{0})\cdot (y-x_{0})+\ga.
\end{equation}
We proceed by contradiction.  Let $w:=\Ga^{u}$ be defined in $G_{0}$. Assume that there exists a point $(x_{0}, t_{0})$ so that 
\begin{equation}\label{contra}
\sup_{B_{1/4}(x_{0}, t_{0})} w\left(\cdot, t_{0}-\frac{1}{16}\right)>  w(x_{0}, t_{0})+Dw\cdot (y-x_{0})+\ga,
\end{equation}
with $\ga$ to be chosen. Without loss of generality, by adding an affine function, we may assume that $(x_{0}, t_{0})=(0,1)$, and $\Ga^{u}(x_{0}, t_{0})=D\Ga^{u}(x_{0}, t_{0})=0$. 

Let $\overline{y}\in \overline{B_{1/4}}$ so that 
\begin{equation*}
w\left(\overline{y}, \frac{15}{16}\right):=\max_{\overline{B_{1/4}}} w\left(\cdot, \frac{15}{16}\right).
\end{equation*}
By \pref{contra}, 
\begin{equation*}
w\left(\overline{y}, \frac{15}{16}\right)> \ga.
\end{equation*}
 
Since $w\left(\cdot, \frac{15}{16}\right)$ is convex and using the definition of $\overline{y}$, this implies that 
\begin{equation*}
w\left(z, \frac{15}{16}\right)> \ga\quad\text{for all}\, z\,\, \text{such that}\,\, z\cdot \overline{y}\geq |\overline{y}|^{2}.
\end{equation*}
In particular, let $\Theta:=\left\{\left(z,\frac{15}{16}\right): z\in B_{1/2}, z\cdot \overline{y}\geq |\overline{y}|^{2}\right\}$. 

Let $\mathcal{Q}:=B_{1/2}\times \left(\frac{15}{16}, 1\right]$. We claim there exists a test function $\vp\in C^{2}(\mathcal{Q})$ which satisfies
\begin{equation}\label{bump}
\begin{cases}
\vp_{t}+\MM^{-}(D^{2}\vp)\geq 0 & \text{in}\quad \mathcal{Q},\\
\vp\geq -\chi_{\Theta}& \text{on}\quad \partial_{p}\mathcal{Q},
\end{cases}
\end{equation}
and $\min \vp(\cdot, 1)\leq -c$ for some universal constant $c$. First, by approximating $-\chi_{\Theta}$ by a smooth function from above and applying the Evans-Krylov theorem \cite{evans_krylov}, there exists a supersolution which is $C^{2}$ satisfying the boundary conditions of \pref{bump}. By the strong maximum principle, there exists a non-constant solution so that $\min \vp(\cdot,1)\leq -c$. Moreover, by compactness, this $c$ can be chosen universally for all $(x_{0}, t_{0})\in Q_{1/4}(0,1)$ by a standard covering argument. This implies that $u+\ga\vp$ satisfies
\begin{equation*}
\begin{cases}
(u+\ga\vp)_{t}+F(D^{2}(u+\ga\vp), x, t)\geq 0 & \text{in}\quad \mathcal{Q},\\
u+\ga\vp\geq 0 & \text{on}\quad \partial_{p}\mathcal{Q},\\
\min_{\mathcal{Q}} (u+\ga\vp)(\cdot,1)\leq -c\ga.
\end{cases}
\end{equation*}
By a similar estimate as in Lemma \ref{muabp}, this implies that $|\mathcal{P}(\mathcal{Q})|\geq c\ga^{d+1}$. Therefore, if we consider covering $\mathcal{Q}$ with a collection of $G_{-2}(x,t)\subseteq G_{0}$, then 
\begin{equation*}
c\ga^{d+1}\leq \sum_{G_{-2}(x,t)\subset G_{0}} |\mathcal{P}(G_{-2}(x,t))|\leq 2|G_{0}|.\end{equation*}
Choosing $\ga$ sufficiently large, depending only on $\la, \La, d$, we obtain a contradiction. Therefore, \pref{upest} holds. 
\end{proof}

By rescaling Lemma \ref{lem:upbnds}, we actually have that if for all $(x,t)\in G_{0}$, 
\begin{equation*}
\frac{|\mathcal{P}(G_{n}(x,t); u)|}{|G_{n}|}\leq 2,
\end{equation*}
and $3^{n}\leq \frac{r}{4}$, then for any point such that $u(x_{0}, t_{0})=\Ga^{u}(x_{0}, t_{0})$, for all $(y,s)\in Q_{r}(x_{0}, t_{0})$, 
 \begin{equation}\label{scaled_upest}
\Ga^{u}(y, s)\leq \Ga^{u}(x_{0}, t_{0})+D\Ga^{u}(x_{0}, t_{0})\cdot (y-x_{0})+\ga r^{2}.\end{equation} 
By sending $r\rightarrow 0$, this implies that $\Ga^{u}_{t}\leq \ga$, and $D^{2}\Ga^{u}\leq \ga Id$ at all contact points where $u=\Ga^{u}$. By the construction of the monotone envelope (in particular, Lemma \ref{repform}), this implies that $\Ga^{u}_{t}\leq \ga$ and $D^{2}\Ga^{u}\leq \ga Id$ everywhere in $G_{0}$. The proof is identical to the proof of Lemma \ref{mereg} which can be found in \cite{cyriluis}. We choose to omit it since it follows verbatim. 

{We highlight that unlike Lemma \ref{mereg}, the upper bound on the time derivatives and Hessian of $\Ga^{u}$ will be \emph{independent of $K_{0}$}. An observation of \cite{asellip} is that it does not seem feasible to obtain an algebraic rate if these upper bounds depend on $K_{0}$. Recall that our goal is to establish an estimate which controls supersolutions from the other side of Lemma \ref{muabp}. Since we plan on performing quantitative analysis, it is important that our estimate is \emph{scale-invariant}. If our estimate depended on $K_{0}$, then by \pref{f3'}, the estimate would depend upon the scaling.}  In general, the upper bounds on the time derivative and the Hessian are controlled by the quantity $\mu(G_{n}(x,t))$. In light of \pref{mu_up}, this is enough to conclude that $\ga$ is independent of $K_{0}$. 

We next show that these upper bounds are actually enough to conclude strict convexity.  

\begin{lem}\label{lem:lbnds}
There exists $\c[lower']=\c[lower'](\la, \La, d)>0$ so that for every $\ve>0$, there exists $n_{1}=n_{1}(\ve, d)<0$ such that if $u\in S(G_{0})$ and $n\leq n_{1}$ satisfies
\begin{equation}\label{mu_sandwich}
1\leq \frac{|\mathcal{P}(G_{n}(x,t); \Ga^{u})|}{|G_{n}|}\leq 2\quad\text{for all}\quad (x,t)\in G_{0},
\end{equation}
then for all $(x_{0}, t_{0})\in Q_{1/4}(0,1)\cap \left\{u=\Ga^{u}\right\}$, for all $(y,s)\in Q_{1/4}(x_{0}, t_{0})$, 
\begin{equation}\label{lbfin}
\Ga^{u}(y,s)\geq \Ga^{u}(x_{0}, t_{0})+D\Ga^{u}(x_{0}, t_{0})\cdot (y-x_{0})+\c[lower'] (t_{0}-s+|y-x_{0}|^{2})-\ve.
\end{equation}
\end{lem}

\begin{proof}
Fix $\ve>0$. Suppose for the purposes of contradiction that \pref{lbfin} does not hold. Therefore, there exists a sequence of $\left(u_{n}, \hat{y}_{n}, \hat{s}_{n}\right)\in S(G_{0})\times G_{0}$ such that $u_{n}$ satisfies \pref{mu_sandwich} for $n$, and $u_{n}$ violates \pref{lbfin} at $(\hat{y}_{n}, \hat{s}_{n})$. Using the convention that $w_{n}:=\Ga^{u_{n}}$, and without loss of generality assuming that $w_{n}\geq 0$ in $G_{0}$ and $w_{n}(0,1)=0$ for each $n$, this amounts to 
\begin{equation}\label{oops}
w_{n}(\hat{y}_{n}, \hat{s}_{n})<c(\hat{s}_{n}+|\hat{y}_{n}|^{2})-\ve
\end{equation}
for $c$ to be chosen.

By \pref{scaled_upest} and \pref{upest}, the family $\left\{w_{n}\right\}$ is equicontinuous and uniformly bounded in $Q_{1/4}(0,1)$. By the Arzela-Ascoli theorem, this implies that there exists a subsequence converging uniformly to a limiting function $w$, with $w$ satisfying 
\begin{equation*}
-w_{t}\leq \ga\quad\text{and}\quad D^{2}w\leq \ga Id\quad\text{almost everywhere}. 
\end{equation*}
 By the Lebesgue Differentiation Theorem and \pref{mu_sandwich}, $w$ also satisfies
\begin{equation*}
1\leq -w_{t}\det D^{2}w\leq 2 \quad\text{almost everywhere}. 
\end{equation*}
Therefore, this yields that $-w_{t}\geq \frac{1}{\ga^{d}}$, and $\det D^{2}w\geq \frac{1}{\ga} Id$ almost everywhere. Since $D^{2}w\leq \ga Id$, this yields that there exists a constant $c_\ga=c(\ga, d)$ so that $D^{2}w\geq c_\ga Id$. 

Consider that by \pref{oops}, since $(\hat{y}_{n}, \hat{s}_{n})\in G_{0}$, there exists a subsequence converging to a point $(\hat{y}, \hat{s})\in G_{0}$ satisfying
\begin{equation*}
w(\hat{y}, \hat{s})<c(\hat{s}+|\hat{y}|^{2})-\ve.
\end{equation*}
However, for $c$ chosen appropriately in terms of $\ga$, this contradicts $-w_{t}\geq \frac{1}{\ga^{d}}$, $D^{2}w\geq \frac{1}{\ga}Id$ almost everywhere.

\end{proof}
Finally, we show that this implies that $u$ will also be strictly convex on the parabolic boundary. 
\begin{thm}\label{growth}
Let $u\in S(G_{1})$. There exists constants $\c[final']=\c[final'](\la, \La, d)$ and $n_{1}=n_{1}(d)<0$ such that if $n\leq n_{1}$ satisfies 
\begin{equation}\label{sandwich}
1\leq \frac{|\mathcal{P}(G_{n}(x,t); \Ga^{u})|}{|G_{n}|}\leq \mu(G_{n}(x,t))\leq 1+3^{n(d+2)}\quad\text{for all}\quad(x,t)\in G_{1},
\end{equation}
then there exists a point $(x_{0}, t_{0})\in \left\{u=\Ga^{u}\right\}\cap G_{n}(0, 9)$ and $(p_{0}, h_{0})\in \mathcal{P}((x_{0}, t_{0}); \Ga^{u})$, so that 
 \begin{equation}\label{genlb}
u(x,t)\geq p_{0}\cdot x+h_{0}+\c[final']\quad\text{for all}\quad \left\{t\leq t_{0} \right\}\cap G_{1}\setminus G_{0}(0,9).
\end{equation}
\end{thm}

\begin{proof}[Proof of Theorem \ref{growth}]
In order to prove \pref{genlb}, it is enough to obtain a lower bound on $\inf_{\partial_{p}G_{0}(0,9)} \Ga^{u}(\cdot, t)$ for $t\leq t_{0}$. We claim there exists $(x_{0}, t_{0})\in G_{n}(0, 9)$ so that $u(x_{0}, t_{0})=\Ga^{u}(x_{0}, t_{0})$. By \pref{sandwich}, for any $(y,s)\in G_{n}(0,9)$, 
\begin{align*}
1&\leq \int_{G_{0}(0,9)}\frac{|\mathcal{P}(G_{n}(x,t); \Ga^{u})|}{|G_{n}|}~dxdt\\
&=\left|\mathcal{P}(G_{n}(y,s); \Ga^{u})\right|+\int_{G_{0}(0,9)\setminus G_{n}(y,s)} \frac{|\mathcal{P}(G_{n}(x,t); \Ga^{u})|}{|G_{n}|}~dxdt\\
&\leq \left|\mathcal{P}(G_{n}(y,s); \Ga^{u})\right|+(1-3^{n(d+2)})(1+3^{n(d+2)}).
\end{align*}
This shows that $|\mathcal{P}(G_{n}(y,s); \Ga^{u})|>0$ for any $(y,s)\in G_{0}$, which implies by Lemma \ref{area} that 
\begin{equation*}
|G_{n}(0,9)\cap \left\{u=\Ga^{u}\right\}|>0. 
\end{equation*}

Let $(x_{0}, t_{0})\in G_{n}(0,9)\cap \left\{u=\Ga^{u}\right\}$, and consider $(p_{0},  h_{0})\in \mathcal{P}((x_{0}, t_{0}); \Ga^{u})$. Let $\tilde{u}(x, t)=u(x, t)-p_{0}\cdot x-h_{0}$. This yields that $\tilde{u}\in S(G_{0}(0,9))$ and $\tilde{u}(x_{0}, t_{0})=\Ga_{\tilde{u}}(x_{0}, t_{0})=0$. Moreover, we have that $(0,0)\in \mathcal{P}((x_{0}, t_{0}); \Ga^{\tilde{u}})$, and $\Ga^{\tilde{u}}\geq 0$ for all $(x,t)\in G_{0}(0,9)\cap \left\{t\leq t_{0}\right\}$. 

By Lemma \ref{lem:lbnds}, letting $\ve=\c[lower']/2$, since $Q_{1/4}(x_{0}, t_{0})\subset G_{0}(0,9)$, this implies that on $\partial_{p}G_{0}(0,9)$, 
\begin{equation*}
u(x,t)\geq \Ga^{u}(x,t)\geq \frac{\c[lower']}{{2}}.
\end{equation*}
Defining $\c[final']:=\frac{\c[lower']}{2}$ completes the proof. 
\end{proof}

For convenience, we also provide a rescaled version of \pref{genlb} which will be used extensively later in the paper. Let $u\in S(G_{m+n+1})$. Let $n\leq n_{1}$ so that 
\begin{equation*}
\al\leq \frac{|\mathcal{P}(G_{n}(x,t); \Ga^{u})|}{|G_{n}|}\leq \mu(G_{n}(x,t))\leq (1+3^{n(d+2)})\al\quad\text{for all}\quad (x,t)\in G_{m+n+1}.
\end{equation*}
There exists a point $(x_{0}, t_{0})\in \left\{u=\Ga^{u}\right\} \cap G_{n}(0, 3^{2(m+n+1)})$ and $(p_{0}, h_{0})\in$\\$\mathcal{P}((x_{0}, t_{0}); \Ga^{u})$ so that, 
\begin{align}\label{scaledlb}
u(x,t)\geq p_{0}\cdot x+h_{0}+&\c[final']\al^{1/(d+1)}3^{2(m+n)}\\
&\quad\text{for all}\quad \left\{t\leq t_{0}\right\}\cap G_{m+n+1}\setminus G_{m+n}(0, 3^{2(m+n+1)})\notag.
\end{align}

\section{The Construction of $\overline{F}$ and the Construction of Approximate Correctors}\label{sec:qualstuff}
We now define the homogenized operator $\overline{F}: \mathbb{S}^{d}\rightarrow \RR$. In addition, we show how one can obtain ``approximate correctors" as in \cite{linhomog} using the quantity $\mu$. For each $M\in \mathbb{S}^{d}$, we say that $w^{\ve}$ is an approximate corrector of \pref{homeq} if there exists $w^{\ve}$ satisfying
\begin{equation}\label{corrector}
\begin{cases}
w^{\ve}_{t}+F(M+D^{2}w^{\ve}, x, t, \om)=\overline{F}(M) & \text{in}\quad Q_{1/\ve},\\
w^{\ve}=0 & \text{on}\quad\partial_{p}Q_{1/\ve},
\end{cases}
\end{equation}
with $\norm{\ve^{2}w^{\ve}}_{L^{\infty}(Q_{1/\ve})}\rightarrow 0$ as $\ve\rightarrow 0$. Once $w^{\ve}$ exists, the qualitative homogenization (the convergence of $u^{\ve}\rightarrow u$ $\PP$-a.s.) follows by a standard perturbed test function argument \cite{evanshom} as shown in \cite{linhomog}.  { In particular, the uniform ellipticity of $\overline{F}$ follows from the existence of approximate correctors.}

\subsection{Identifying $\overline{F}$}
We identify $\overline{F}(M)$ for each fixed $M\in \mathbb{S}^{d}$. First, we establish a lemma which states that $\mu$ is Lipschitz continuous with respect to the right hand side $\ell$.

{

\begin{lem}\label{lem:lip}
There exists a $C(\lambda,\Lambda,d,M,K_0) > 0$ such that
\begin{equation}
\label{lipschitzmu}
0 \geq \mu(Q, \omega, \ell + s, M) - \mu(Q, \omega, \ell, M) \geq - C |Q| s,
\end{equation}
for all $s \in [0,1]$.
\end{lem}

\begin{proof}
The left inequality follows from the comparison principle for viscosity solutions, since $S(Q, \om, \ell+s, M)\subseteq S(Q, \om, \ell, M)$. To obtain the right inequality, let $u \in S(Q,\omega,\ell,M)$, and define $u^s(x,t):= u(x,t) + s t$  which lies in $S(Q,\omega,\ell+s,M)$.  Let $w^s$ denote the monotone envelope of $u^s$.  We note that $|w^s_t|, |D^2 w^s| \leq C(K_{0}, \ell+s, M)$ on the contact set $\{ u^s = w^s \}$ by Lemma \ref{mereg} and Lemma \ref{area}. Therefore, by the area formula, this implies that 
\begin{align*}
|\mathcal{P}(Q; w^s)| &= \int_{\{ u^s = w^s \}\cap Q} - u^s_t \det D^2 u^s \,dx,\\
&\geq \int_{\{ u = w \}\cap \left\{u_{t}\leq -s\right\}\cap Q} - u^s_t \det D^2 u^s \,dx,\\
&\geq \int_{\{ u=w\}\cap Q}-u_{t}\det D^{2}u-Cs|Q|\\
&=|\mathcal{P}(Q; w)|-Cs|Q|. 
\end{align*}
By taking the supremum over $u\in S(Q, \om, \ell, M)$, this yields \pref{lipschitzmu}. 
\end{proof}

}

\begin{lem}\label{lemlim}
Let $M\in \mathbb{S}^{d}$. For every $n\in \NN$, the map
\begin{equation*}
\ell\rightarrow \EE[\mu(G_{n}, \om,\ell, M)]\quad\text{is continuous and nonincreasing.} 
\end{equation*}
Similarly, the map 
\begin{equation*}
\ell\rightarrow \EE[\mu^{*}(G_{n}, \om, \ell, M)]\quad\text{is continuous and nondecreasing.}
\end{equation*}

In addition, there exists $\overline{\ell}(M)\in \RR$ so that $\PP$-a.s. in $\om$, 
\begin{align}\label{eqest}
&\lim_{n\rightarrow\infty} \mu(G_{n}, \om, \overline{\ell}(M), M)=\lim_{n\rightarrow \infty} \EE[\mu(G_{n}, \om, \overline{\ell}(M), M)]\notag\\
&=\lim_{n\rightarrow \infty}\EE[\mu^{*}(G_{n}, \om, \overline{\ell}(M), M)]=\lim_{n\rightarrow\infty} \mu^{*}(G_{n}, \om,\overline{\ell}(M), M).
\end{align}
\end{lem}

\begin{proof}
The Lipschitz continuity and monotonicity follow from Lemma \ref{lem:lip}. By \pref{bndest}, $\EE[\mu(G_{n}, \om, \ell)]=0$ for all $\ell\geq K_{0}(1+|M|)$. In particular, this implies that 
\begin{equation*}
\lim_{n\rightarrow\infty} \EE[\mu(G_{n}, \om, \ell)]=0\quad\text{for all}\quad\ell\geq K_{0}(1+|M|).
\end{equation*}
 Similarly, 
 \begin{equation*}
 \lim_{n\rightarrow\infty} \EE[\mu^{*}(G_{n}, \om, \ell)]=0\quad\text{for all}\quad \ell\leq -K_{0}(1+|M|). 
 \end{equation*}

Using the monotonicity in $\ell$ and \pref{bndest}, there exists a choice of $\overline{\ell}$ so that \\$\displaystyle \lim_{n\rightarrow \infty} \EE[\mu(G_{n}, \om,\overline{\ell})]=\lim_{n\rightarrow \infty}\EE[\mu^{*}(G_{n}, \om, \overline{\ell})]$. The outer equalities of \pref{eqest} hold in light of the ergodicity assumption \pref{f1'} and the subadditive ergodic theorem.
\end{proof}

Using Lemma \ref{lemlim}, we define
\begin{equation}
\overline{F}(M):=\overline{\ell}(M).
\end{equation}
We will now show that $\overline{F}(M)$ agrees with the effective operator constructed in \cite{linhomog}, and thus the uniqueness follows. To do this, it is enough to show that solutions $w^{\ve}$ of \pref{corrector} exist and satisfy the desired limiting behavior. 

\subsection{A Qualitative Homogenization Argument}
The construction of approximate correctors \pref{corrector} follows in two steps. First, we show that for any $M\in \mathbb{S}^{d},$ it is impossible for both $E(\overline{\ell}(M), M):=\displaystyle \lim_{n\rightarrow\infty} \mu(G_{n}, \om, \overline{\ell}(M), M)$ and $E^{*}(\overline{\ell}, M):=\displaystyle \lim_{n\rightarrow\infty} \mu^{*}(G_{n}, \om, \overline{\ell}(M), M)$ to be positive. Applying Lemma \ref{muabp} allows us to conclude. 

For convenience, we provide a precise statement of the Harnack inequality for parabolic equations, as can be found in \cite{wangreg1, cyriluis}. We will use the notation of this theorem in the future. 

\begin{thm}[Harnack Inequality]\label{harnack}
Let $u$ be nonnegative, and { $u_{t}+\MM^{+}(D^{2}u)\geq -|f|$ and $u_{t}+\MM^{+}(D^{2}u)\leq |f|$}. Then there exists a universal $C=C(\la, \La, d)$ so that 
\begin{equation*}
\sup_{\tilde{Q}} u\leq C(\inf_{Q_{\rho^{2}}} u+\norm{f}_{L^{d+1}(Q_{1})})
\end{equation*}
where $\tilde{Q}:=B_{\frac{\rho^{2}}{2\sqrt{2}}}\times (-\rho^{2}+\frac{3}{8} \rho^{4}, -\rho^{2}+\frac{1}{2}\rho^{4})\subseteq Q_{1}$, and $\rho=\rho(\la, \La, d)$. 
\end{thm}

{
The Harnack inequality implies that $E$ and $E^*$ must vanish when they are equal.
}

\begin{lem}\label{both0}
Fix $M\in \mathbb{S}^{d}$. If $\ell\in \RR$ such that
\begin{equation}
\lim_{n\rightarrow\infty} \EE[\mu(G_{n}, \om, \ell, M)]=E(\ell, M)=E^{*}(\ell, M)=\lim_{n\rightarrow\infty}\EE[\mu(G_{n}, \om^{*}, -\ell, M)],
\end{equation}
then $E(\ell, M)=E^{*}(\ell, M)=0$. 
\end{lem}
\begin{proof}
We drop the dependence on $M$ since it is fixed throughout the proof. Suppose that both $E(\ell)=E^{*}(\ell):=\al>0$. By the subadditive ergodic theorem, there exists a choice of $m$ sufficiently large so that for all $(x,t)\in G_{m+n}$ with $n$ large to be chosen, 
{
\begin{equation*}
\frac{1}{2} \al\leq \frac{|\mathcal{P}(G_{m}(x,t); \Ga^{u})|}{|G_{m}|}\leq \mu(G_{m}, \om, \ell)\leq 2 \al.
\end{equation*}
}
Without loss of generality, we assume that $m=0$. By Theorem \ref{growth} rescaled, choosing $n$ sufficiently large, and after an affine transformation, there exists a function $u$ so that 
\begin{equation}\label{eqsup_l}
u_{t}+F(D^{2}u, x, t, \om)=\ell \quad\text{in}\quad G_{n}(0, 3^{2(n+1)})
\end{equation}
and $(x_{0},t_{0})\in G_{0}(0, 3^{2(n+1)})$ so that 
\begin{equation}\label{fromregsup}
u\geq u(x_{0}, t_{0})+C3^{2n}\al^{1/(d+1)}\quad\text{on}\quad \partial_{p}G_{n}(0, 3^{2(n+1)})\cap \left\{t\leq t_{0}\right\},
\end{equation}
and
\begin{equation*}
\inf_{G_{n}(0, 3^{2(n+1)})\cap\left\{t\leq t_{0}\right\}} u=\inf_{G_{0}(0, 3^{2(n+1)})\cap\left\{t\leq t_{0}\right\}}u=u(x_{0}, t_{0})=0.
\end{equation*}
This is done by extracting $u'\in S(G_{n+1}, \om)$ such that \pref{genlb} holds. Upon an affine transformation and solving \pref{eqsup_l} with $u=u'$ on $\partial_{p}G_{n}(0,3^{2(n+1)})$, we have the claim. 
Similarly, there exists $u^{*}$ satisfying
\begin{equation}\label{eqsub_l}
u^{*}_{t}+F(D^{2}u^{*}, x, t, \om^{*})=-\ell\quad\text{in}\quad G_{n}(0, 3^{2(n+1)})
\end{equation}
and for some $(x_{0}^{*}, t_{0}^{*})\in G_{0}(0, 3^{2(n+1)})$,
\begin{equation}\label{fromregsub}
u^{*}\geq u^{*}(x_{0}, t_{0})+C3^{2n}\al^{1/(d+1)}\quad\text{on}\quad \partial_{p}G_{n}(0, 3^{2(n+1)})\cap \left\{t\leq t^{*}_{0}\right\},
\end{equation}
and
\begin{equation*}
\inf_{G_{n}(0, 3^{2(n+1)})\cap\left\{t\leq t_{0}^{*}\right\}} u^{*}= \inf_{G_{0}(0, 3^{2(n+1)})\cap \left\{t\leq t_{0}^{*}\right\}} u^{*}=u^{*}(x_{0}^{*}, t_{0}^{*})=0.
\end{equation*}

Let $\overline{t}=\min\left\{t_{0}, t_{0}^{*}\right\}$. Notice that $w:=u+u^{*}$ satisfies 
\begin{align*}
w_{t}+\MM^{+}(D^{2}w)&\geq u_{t}+u^{*}_{t}+F(D^{2}u, x, t, \om)+F(D^{2}u^{*}, x, t, \om^{*})\\
&=0\quad\text{in}\quad G_{n}(0, 3^{2(n+1)}),
\end{align*}
and 
\begin{equation*}
w\geq C3^{2n}\al^{1/(d+1)}\quad\text{on}\quad \partial_{p}G_{n}(0, 3^{2(n+1)})\cap\left\{t\leq \overline{t}\right\}.
\end{equation*}

By the Alexandrov-Backelman-Pucci-Krylov-Tso estimate \cite{wangreg1, cyriluis}, this implies that 
\begin{equation}\label{lowercomp}
w\geq C3^{2n}\al^{1/(d+1)}\quad\text{in}\quad G_{n}(0, 3^{2(n+1)})\cap\left\{t\leq \overline{t}\right\}.
\end{equation}
Let $s$ be defined as the smallest integer such that $\rho^{2}3^{s}\geq \sqrt{d}$ where $\rho$ is defined in the Harnack inequality (Theorem \ref{harnack}). We may assume that $s\leq n$ by choosing $n$ larger if necessary. We observe that in $G_{s}(0, 3^{2(n+1)})$, $u, u^{*}$ also each satisfy
\begin{equation*}
u_{t}+\MM^{+}(D^{2}u)\geq -|\ell|-K_{0}\quad\text{and}\quad K_{0}+|\ell| \geq u_{t}+\MM^{-}(D^{2}u),
\end{equation*}
and 
\begin{equation*}
u^{*}_{t}+\MM^{+}(D^{2}u^{*})\geq -|\ell|-K_{0} \quad\text{and}\quad |\ell|+K_{0} \geq u^{*}_{t}+\MM^{-}(D^{2}u^{*}).
\end{equation*}

Since $\inf_{G_{0}(0, 3^{2(n+1)})}u=\inf_{G_{0}(0, 3^{2(n+1)})}u^{*}=0$, and 
\begin{equation*}
G_{0}(0, 3^{2(n+1)})\subseteq Q_{\rho^{2}3^{s}}(0, 3^{2(n+1)})
\end{equation*}
by our choice of $s$, this implies by the Harnack inequality that there exists $C=C(\la, \La, d, \ell, K_{0})$ so that 
\begin{equation*}
\sup_{\tilde{Q}}u \leq C3^{2s}\quad\text{and}\quad \sup_{\tilde{Q}} u^{*}\leq C3^{2s},
\end{equation*}
where $\tilde{Q}\subseteq G_{s}(0, 3^{2(n+1)})$ is a rescaled version of $\tilde{Q}$ defined in Theorem \ref{harnack}. Thus, there exists $C=C(\la, \La, d, \ell, K_{0})>0$ so that 
\begin{equation*}
w\leq C3^{2s}\quad \text{in}\quad \tilde{Q}\subseteq G_{s}(0, 3^{2(n+1)}).
\end{equation*}
By choosing $n$ sufficiently large, depending on $\ell, K_{0}, \al$, we obtain a contradiction with \pref{lowercomp}. Therefore, $\al=0$. 
\end{proof}

We next show that $w^{\ve}$ solving \pref{corrector} has the desired decay with this definition of $\overline{F}(M)$. Letting $\ve=3^{-n}$, we relabel \pref{corrector} as 
\begin{equation}\label{newcorrector}
\begin{cases}
w^{n}_{t}+F(M+D^{2}w^{n}, x, t, \om)=\overline{F}(M) & \text{in}\quad G_{n},\\
w^{n}=0 & \text{on}\quad \partial_{p}G_{n},
\end{cases}
\end{equation}
and we want to show that $\norm{3^{-2n}w^{n}}_{L^{\infty}(G_{n})}\rightarrow 0$ as $n\rightarrow \infty$. 

Consider that since $E(\overline{F}(M),M)=E^{*}(\overline{F}(M), M)=0$, this implies that almost surely, 
\begin{equation*}
 \lim_{n\rightarrow\infty} \mu(G_{n}, \om)=0=\lim_{n\rightarrow\infty} \mu^{*}(G_{n}, \om).
 \end{equation*} By Lemma \ref{muabp} and \pref{newcorrector}, this implies that 
\begin{equation*}
0\leq \inf_{G_{n}} 3^{-2n}w^{n}+\c[ptf]\mu(G_{n}, \om)^{1/(d+1)},
\end{equation*}
and 
\begin{equation*}
0\geq \sup_{G_{n}} 3^{-2n}w^{n}-\c[ptf]\mu^{*}(G_{n}, \om)^{1/(d+1)}.
\end{equation*}
Taking $n\rightarrow \infty$, this yields
\begin{equation}\label{quaddecay}
\lim_{n\rightarrow\infty}\norm{3^{-2n}w^{n}}_{L^{\infty}(G_{n})}\leq \lim_{n\rightarrow\infty} \max\left\{\mu(G_{n}, \om)^{1/(d+1)}, \mu^{*}(G_{n}, \om)^{1/(d+1)}\right\}=0,
\end{equation}
as desired.

\section{A Rate of Decay on the Second Moments}\label{sec:decaym}
In this section, we obtain a rate of decay on the second moments of $\mu$. {The approach of this section closely follows that of \cite{asellip}}. As before, we suppress the dependence on $M$. We simplify the notation by adopting the following conventions. Let
\begin{equation*}
E_{n}(\ell)=\EE[\mu(G_{n}, \om, \ell)]\quad\text{and}\quad E^{*}_{n}(\ell)=\EE[\mu^{*}(G_{n}, \om, \ell)]=\EE[\mu(G_{n}, \om^{*}, -\ell)].
\end{equation*}
Also, let 
\begin{equation*}
J_{n}(\ell)=\EE[\mu(G_{n}, \om, \ell)^{2}]\quad\text{and}\quad J^{*}_{n}(\ell)=\EE[\mu^{*}(G_{n}, \om, \ell)^{2}]=\EE[\mu(G_{n}, \om^{*}, -\ell)^{2}].
\end{equation*}

{ Our next lemma shows that if the variance of $\mu$ and $\mu^*$ are not decaying, then their expectations must be close to zero.}
The proof resembles the argument for Lemma \ref{both0}, but avoids the dependence on $K_{0}$. 
\begin{lem}\label{decaylem}
Suppose that there exists $m,n\in \mathbb{N}$ and $\eta, \ga>0$ such that 
\begin{equation}\label{notenough}
0< J_{m}(\ell-\ga)\leq (1+\eta)E^{2}_{m+n}(\ell-\ga),
\end{equation}
and
\begin{equation}\label{notenough*}
0< J^{*}_{m}(-\ell+\ga) \leq (1+\eta)E^{*2}_{m+n}(-\ell+\ga).
\end{equation}
Then there exists $n_{0}=n_{0}(\la, \La, d)$ and $\eta_{0}=\eta_{0}(\la, \La, d)$ so that for all $n\geq n_{0}$ and for all $\eta\leq \eta_{0}$, 
\begin{equation}\label{seriousdecay}
J_{m+n}(\ell-\ga)+J^{*}_{m+n}(-\ell+\ga)\leq C \ga^{2(d+1)}.
\end{equation}
\end{lem}

\begin{proof}
Without loss of generality, we assume that $\ell=0$, $m=0$. First, we claim that there exists a choice of environment $\om$ such that $\mu(G_{n}, \om)$ and $\mu(G_{0}(x,t), \om)$ is approximately constant for all $(x,t)\in G_{n}$. 

Fix $\delta>0$. There exists $\eta=\eta(\delta)$ such that if \pref{notenough} and \pref{notenough*} hold for this $\eta$, there exists an $\om$ so that for all $(x,t)\in G_{n}$, 
\begin{equation}\label{const_sand1}
(1-\delta)E_{n}(-\ga) \leq \mu(G_{n}, \om, -\ga)\leq \mu(G_{0}(x,t), \om, -\ga)\leq (1+\delta)E_{n}(-\ga),
\end{equation}
and similarly for the lower quantity, 
\begin{equation}\label{const_sand2}
(1-\delta)E^{*}_{n}(\ga) \leq \mu^{*}(G_{n}, \om, \ga)\leq \mu^{*}(G_{0}(x,t), \om, \ga)\leq (1+\delta)E^{*}_{n}(\ga).
\end{equation}

Applying Chebyshev's inequality, we have that for any $(x,t)\in G_{n}$, 
\begin{align*}
&\PP[\mu(G_{0}(x,t), \om, -\ga)\geq (1+\delta)E_{n}(-\ga)]\\
&\leq \PP[\mu(G_{0}(x,t), \om, -\ga)-E_{n}(-\ga)\geq \delta E_{n}(-\ga)]\\
&\leq \PP[[\mu(G_{0}(x,t), \om, -\ga)-E_{n}(-\ga)]^{2}\geq \delta^{2}E_{n}^{2}(-\ga)]\\
&\leq \frac{1}{\delta^{2}E_{n}^{2}(-\ga)} \EE[[\mu(G_{0}(x,t), \om, -\ga)-E_{n}(-\ga)]^{2}]\\
&\leq \frac{1}{\delta^{2}E_{n}^{2}(-\ga)}\left[J_{0}(-\ga)-E_{n}^{2}(-\ga)\right]\\
&\leq \eta \delta^{-2},
\end{align*}
where the last inequality follows from \pref{notenough}.

Similarly, 
\begin{align*}
&\PP[\mu(G_{n}, \om, -\ga)<(1-\delta)E_{n}(-\ga)]\\
&\leq \PP[(\mu(G_{n}, \om, -\ga)-E_{n}(-\ga))^{2}\geq \delta^{2} E_{n}(-\ga)^{2}]\\
&\leq\frac{1}{\delta^{2}E_{n}(-\ga)^{2}} \EE[(\mu(G_{n}, \om, -\ga)-E_{n}(-\ga))^{2}]\\
&\leq \frac{1}{\delta^{2}E_{n}(-\ga)^{2}}\left( \EE[\mu(G_{n}, \om, -\ga)^{2}]-E_{n}(-\ga)^{2}\right)\\
&\leq \eta \delta^{-2}.
\end{align*}

By identical arguments, 
\begin{equation*}
\PP[\mu^{*}(G_{0}(x,t), \om, \ga)\geq (1+\delta)E^{*}_{n}(\ga)]\leq \eta \delta^{-2},
\end{equation*}
and
\begin{equation*}
\PP[\mu^{*}(G_{n}, \om, \ga)<(1-\delta)E^{*}_{n}(\ga)]\leq \eta \delta^{-2}.
\end{equation*}
By a union bound, this implies that
\begin{equation}
\PP\left[\pref{const_sand1}, \pref{const_sand2}\quad\text{hold for all $(x,t)\in G_{n}$}\right]\geq 1-4\eta\delta^{-2},
\end{equation}
so by choosing $\eta\leq \frac{1}{4}\delta^{2}$, this has positive probability. Let $\om\in \Om$ be an element of this set, which implies $\om$ satisfies \pref{const_sand1} and \pref{const_sand2} for all $(x,t)\in G_{n}$. Using this particular $\om$, we next show that there exist constants $c,C$ and $s\in \NN$ which only depend on $\la, \La, d$ so that 
\begin{equation}\label{somerelation}
c\left[E_{n}(-\ga)+E^{*}_{n}(\ga)-C\ga^{d+1}\right]\leq (1+\delta)3^{-2(n-s)(d+1)}\left[E_{n}(-\ga)+E^{*}_{n}(\ga)\right].
\end{equation}

Consider that by Theorem \ref{growth}, similar to the proof of Lemma \ref{both0}, there exists $n=n(d, \la, \La)$ and $u, u^{*}\in C(G_{n}(0, 3^{2(n+1)}))$ so that 
\begin{equation*}
u_{t}+F(D^{2}u, x, t,\om)=-\ga\quad\text{in}\quad G_{n}(0, 3^{2(n+1)}),
\end{equation*}
with 
\begin{equation*}
\inf_{\partial_{p}G_{n}(0, 3^{2(n+1)})\cap \left\{t\leq t_{0}\right\}}u(x,t)\geq C3^{2n}E_{n}(-\ga)^{1/(d+1)},
\end{equation*}
and
\begin{equation*}
\inf_{G_{0}(0, 3^{2(n+1)})}u=\inf_{G_{n}(0, 3^{2(n+1)})} u=0.
\end{equation*}

Similarly, $u^{*}$ satisfies 
\begin{equation*}
u^{*}_{t}+F(D^{2}u^{*}, x, t, \om^{*})=-\ga\quad\text{in}\quad G_{n}(0, 3^{2(n+1)}),
\end{equation*} 
with
\begin{equation*}
\inf_{\partial_{p}G_{n}(0, 3^{2(n+1)})\cap \left\{t\leq t^{*}_{0}\right\}} u^{*}(x,t) \geq C3^{2n}E^{*}_{n}(\ga)^{1/(d+1)}
\end{equation*}
and
\begin{equation*}
\inf_{G_{0}(0, 3^{2(n+1)})} u^{*}=\inf_{G_{n}(0, 3^{2(n+1)})} u^{*}=0.
\end{equation*}
Let $\tilde{t}=\min\left\{t_{0}, t^{*}_{0}\right\}$. We note that the function $u+u^{*}$ satisfies that 
\begin{equation*}
u+u^{*}\geq C 3^{2n}(E_{n}(-\ga)^{1/(d+1)}+E^{*}_{n}(\ga)^{1/(d+1)})\quad\text{on}\quad \partial_{p}G_{n}(0, 3^{2(n+1)})\cap \left\{t\leq \tilde{t}\right\},
\end{equation*}
and 
\begin{equation*}
(u+u^{*})_{t}+\MM^{+}(D^{2}(u+u^{*}))\geq -2\ga\quad\text{in}\quad G_{n}(0, 3^{2(n+1)}).
\end{equation*}

By the Alexandrov-Backelman-Pucci-Krylov-Tso estimate \cite{wangreg1, cyriluis}, this implies that 
\begin{align}\label{positivesum}
u+u^{*}\geq c 3^{2n}\left[E_{n}(-\ga)^{1/(d+1)}+E^{*}_{n}(\ga)^{1/(d+1)}\right]&-C3^{2n}\ga\\
&\text{in}\quad G_{n}(0, 3^{2(n+1)})\cap \left\{t\leq \tilde{t}\right\}\notag.
\end{align}

Next, consider the solutions $w, \tilde{w}$ solving
\begin{equation*}
\begin{cases}
w_{t}+F(D^{2}w, x, t, \om)=-\ga & \text{in}\quad G_{s}(0, 3^{2(n+1)}),\\
w=0 & \text{on}\quad \partial_{p}G_{s}(0, 3^{2(n+1)}),
\end{cases}
\end{equation*}
and 
\begin{equation*}
\begin{cases}
w^{*}_{t}+F(D^{2}w^{*}, x, t, \om^{*})=-\ga & \text{in}\quad  G_{s}(0, 3^{2(n+1)}),\\
w^{*}=0 & \text{on}\quad \partial_{p}G_{s}(0, 3^{2(n+1)}),
\end{cases}
\end{equation*}
with $s$ to be chosen such that $s\leq n$. 

We have that 
\begin{equation*}
w+w^{*}=0\quad\text{on}\quad \partial_{p}G_{s}(0, 3^{2(n+1)}),
\end{equation*}
and 
\begin{equation*}
(w+w^{*})_{t}+\MM^{-}(D^{2}(w+w^{*}))\leq -2\ga\leq 0\quad\text{in}\quad G_{s}(0, 3^{2(n+1)}).
\end{equation*}
This implies that
\begin{equation}\label{negativesum}
w+w^{*}\leq 0\quad\text{in}\quad G_{s}(0, 3^{2(n+1)}).
\end{equation}

Combining \pref{positivesum} and \pref{negativesum}, we have that for all $(x,t)\in G_{s}(0, 3^{2(n+1)})\cap \left\{t\leq \overline{t}\right\}$, 
\begin{equation}\label{seriousbnd}
w(x,t)-u(x,t)+w^{*}(x,t)-u^{*}(x,t)\leq C3^{2n}\ga-c 3^{2n}(E_{n}(-\ga)^{1/(d+1)}+E^{*}_{n}(\ga)^{1/(d+1)}).\end{equation}

Notice that   
\begin{equation*}
w-u\leq 0\quad\text{on}\quad\partial_{p}G_{s}(0, 3^{2(n+1)}),
\end{equation*}
and in $G_{s}(0, 3^{2(n+1)})$, 
\begin{equation*}
(w-u)_{t}+\MM^{+}(D^{2}(w-u))\geq 0\geq (w-u)_{t}+\MM^{-}(D^{2}(w-u)).
\end{equation*}

This implies that $w-u\leq 0$ in $G_{s}(0, 3^{2(n+1)})$. Consider the Harnack inequality (Theorem \ref{harnack}) applied to $u-w\geq 0$. By the Harnack inequality rescaled in $G_{s}(0, 3^{2(n+1)})$, (where $\tilde{Q}$ corresponds to the rescaled $\tilde{Q}$), 
\begin{equation*}
\sup_{\tilde{Q}} (u-w)\leq C \inf_{Q_{\rho^{2}3^{s}(0, 3^{2(n+1)})}} (u-w).
\end{equation*}
This implies that 
\begin{equation*}
-\sup_{\tilde{Q}} (u-w)\geq -C\inf_{Q_{\rho^{2}3^{s}(0, 3^{2(n+1)})}} (u-w),
\end{equation*}
which yields
\begin{equation}\label{reverse}
\inf_{\tilde{Q}} (w-u)\geq C \sup_{Q_{\rho^{2}3^{s}(0, 3^{2(n+1)})}} (w-u).
\end{equation}

Choose $s$ so that $G_{0}(0, 3^{2(m+1)})\subseteq Q_{\rho^{2}3^{s}}(0, 3^{2(m+1)})$. Since \pref{seriousbnd} holds for all $(x,t)\in G_{s}(0, 3^{2(n+1)})\cap\left\{t\leq \tilde{t}\right\}$ and $\tilde{Q}\subseteq G_{s}(0, 3^{2(n+1)})\cap\left\{t\leq \tilde{t}\right\}$, we may assume without loss of generality that 
\begin{equation*}
\inf_{\tilde{Q}} (w-u)\leq \frac{1}{2}[C3^{2n}\ga-c 3^{2n}(E_{n}(-\ga)^{1/(d+1)}+E^{*}_{n}(\ga)^{1/(d+1)})].
\end{equation*}
 (If not, then we repeat this analysis for $w^{*}-u^{*}$.) By \pref{reverse}, this implies that in $Q_{\rho^{2}3^{s}}(0, 3^{2(n+1)})$, 
\begin{equation*}
w-u\leq C[3^{2n}\ga-c 3^{2n}(E_{n}(-\ga)^{1/(d+1)}+E^{*}_{n}(\ga)^{1/(d+1)})].
\end{equation*}

In particular, we have that 
\begin{align*}
\inf_{Q_{\rho^{2}3^{s}}(0, 3^{2(n+1)})} w\leq \inf&_{Q_{\rho^{2}3^{s}}(0, 3^{2(n+1)})} u\\
&+c\left[C3^{2n}\ga-3^{2n}(E_{n}(-\ga)^{1/(d+1)}+E^{*}_{n}(\ga)^{1/(d+1)})\right].
\end{align*}

Since $(x_{0}, t_{0})\in G_{0}(0, 3^{2(n+1)})\subseteq Q_{\rho^{2}3^{s}}(0, 3^{2(n+1)})$, this implies that 
\begin{equation*}
\inf_{Q_{\rho^{2}3^{s}}(0, 3^{2(n+1)})} u=0, 
\end{equation*}
which yields
\begin{equation}\label{whoa}
\inf_{Q_{\rho^{2}3^{s}}(0, 3^{2(n+1)})} w\leq c\left[C3^{2n}\ga-3^{2n}(E_{n}(-\ga)^{1/(d+1)}+E^{*}_{n}(\ga)^{1/(d+1)})\right].
\end{equation}

By Lemma \ref{muabp}, since $w=0$ on $\partial_{p}G_{s}(0, 3^{2(n+1)})$, 
\begin{align*}
0&\leq \inf_{G_{s}(0, 3^{2(n+1)})} w+\c[ptf]3^{2s}\mu(G_{s}(0, 3^{2(n+1)}), \om, -\ga)^{1/(d+1)}\\
&\leq \inf_{Q_{\rho^{2}3^{s}}(0, 3^{2(n+1)})} w+\c[ptf]3^{2s}\mu(G_{s}(0, 3^{2(n+1)}), \om, -\ga)^{1/(d+1)}.
\end{align*}
By \pref{whoa}, this implies  
\begin{align*}
c3^{2(n-s)(d+1)}[E_{n}(-\ga)^{1/(d+1)}+&E^{*}_{n}(\ga)^{1/(d+1)}-C\ga]^{d+1}\leq \mu(G_{s}(0, 3^{2(n+1)}), \om, -\ga)\\
&\leq \aint_{G_{s}(0, 3^{2(n+1)})}\mu(G_{0}(x,t), \om)\,dxdt\\
&\leq (1+\delta)E_{n}(-\ga)\leq (1+\delta)[E_{n}(-\ga)+E_{n}^{*}(\ga)].
\end{align*}

This yields
\begin{align*}
3^{2(n-s)(d+1)}c(E_{n}(-\ga)+E^{*}_{n}(\ga)-C\ga^{d+1})\leq (1+\delta)\left[E_{n}(-\ga)+E^{*}_{n}(\ga)\right],
\end{align*}
which is equivalent to \pref{somerelation}.

To conclude, we just need to choose $\delta$, $\eta$ and show there is an $n$ sufficiently large to obtain \pref{seriousdecay}.

Rearranging yields 
\begin{equation*}
[1-3^{-2(n-s)(d+1)}-\delta 3^{-2(n-s)(d+1)}][E_{n}(-\ga)+E^{*}_{n}(\ga)]\leq C \ga^{d+1}.
\end{equation*}
Choosing $\delta:= 3^{-2s(d+1)}$, and $\eta\leq \frac{1}{4}3^{-4s(d+1)}$ yields a choice of $\om\in \Om$ such that \pref{const_sand1} and \pref{const_sand2} hold, and
\begin{equation*}
[1-3^{-2(n-s)(d+1)}-3^{-2n(d+1)}][E_{n}(-\ga)+E^{*}_{n}(\ga)]\leq C \ga^{d+1}.
\end{equation*}

For any $n\geq 2s$, we have that 
\begin{equation*}
E_{n}(-\ga)+E^{*}_{n}(\ga)\leq C[1-3^{-2s(d+1)}-3^{-4s(d+1)}]^{-1}\ga^{d+1}=C\ga^{d+1}.
\end{equation*}
This implies that 
\begin{equation*}
J_{n}(-\ga)+J^{*}_{n}(\ga)\leq (1+\eta)\left[E_{n}(-\ga)^{2}+E^{*}_{n}(\ga)^{2}\right]\leq C\ga^{2(d+1)}
\end{equation*}
as asserted.
\end{proof}

{We next show how the finite range of dependence assumption \pref{f1'} yields a relation between $J_{m+n}(\ell)$ and $J_{m}(\ell)$ for $n>0$. }
{
\begin{lem}\label{lem:iid}
There exists $\c[mixdecay]=\c[mixdecay](d)$ such that for any $\ell$, and for any $m, n\geq 0$, 
\begin{equation}\label{vardecay}
J_{m+n}(\ell)\leq E_{m}^{2}+\frac{\c[mixdecay]}{3^{n(d+2)}}J_{m}(\ell).
\end{equation}
Similarly, 
\begin{equation}\label{vardecay*}
J^{*}_{m+n}(-\ell)\leq E^{*2}_{m}+\frac{\c[mixdecay]}{3^{n(d+2)}}J^{*}_{m}(-\ell).
\end{equation}
\end{lem}

\begin{proof}
Since $\ell$ plays no role, we suppress its dependence. Consider that $G_{m+n}=\bigcup_{i=1}^{3^{n(d+2)}} G_{m}^{i}$ for some choice of enumeration of cubes $\left\{G_{m}^{i}\right\}$. Therefore, for each $u\in S(G_{m+n}, \om)$, 
\begin{align*}
|\mathcal{P}(G_{m+n}; u)|^{2}&= \left(\sum_{i=1}^{3^{n(d+2)}} \left|\mathcal{P}(G^{i}_{m}; u)\right|\right)^{2}\\
&=\sum_{i} |\mathcal{P}(G^{i}_{m}; u)|^{2}+\sum_{i}\sum_{j\neq i}|\mathcal{P}(G^{i}_{m}; u)||\mathcal{P}(G^{j}_{m}; u)|\\
&=\sum_{i=1}^{3^{n(d+2)}} |\mathcal{P}(G^{i}_{m}; u)|^{2}+\sum_{i=1}^{3^{n(d+2)}} \left[\sum_{d[G^{i}_{m}, G^{j}_{m}]>1}|\mathcal{P}(G^{i}_{m}; u)||\mathcal{P}(G^{j}_{m}; u)|\right.\\
&\left.+\sum_{d[G^{i}_{m}, G^{j}_{m}]\leq1}|\mathcal{P}(G^{i}_{m}; u)||\mathcal{P}(G^{j}_{m}; u)|\right].
\end{align*}

This implies that 
\begin{align*}
\mu(G_{m+n}, \om)^{2}&\leq \frac{1}{3^{2n(d+2)}} \sum_{i=1}^{3^{n(d+2)}} (\mu(G^{i}_{m}, u))^{2}\\
&+\frac{1}{3^{2n(d+2)}}\sum_{i=1}^{3^{n(d+2)}} \left[\sum_{d[G^{i}_{m}, G^{j}_{m}]>1}\mu(G^{i}_{m}, \om)\mu(G^{j}_{m}, \om)\right.\\
&\left.+\sum_{d[G^{i}_{m}, G^{j}_{m}]\leq 1}\mu(G^{i}_{m}, \om)\mu(G^{j}_{m}, \om)\right].
\end{align*}

For each $i$ fixed, if $d[G^{i}_{m}, G^{j}_{m}]>1$, by \pref{iid}, stationarity, and Lemma \ref{bnds},
\begin{align*}
\EE[\mu(G^{i}_{m}, \om)\mu(G^{j}_{m}, \om)]&= E_{m}^{2}.
\end{align*}
If $d[G^{i}_{m}, G^{j}_{m}]\leq 1$, then by the Cauchy-Schwartz inequality and stationarity, 
\begin{equation*}
\EE[\mu(G^{i}_{m}, \om)\mu(G^{j}_{m}, \om)]\leq \EE[\mu(G_{m}, \om)^{2}]=J_{m}. 
\end{equation*}
For any fixed $i$, the number of cubes so that $d[G^{i}_{m}, G^{j}_{m}]\leq 1$ is at most $3^{d+1}$. Therefore, after taking expectation of both sides, summing over $i=1, \ldots, 3^{n(d+2)}$ copies, this yields that
\begin{align*}
J_{m+n} &\leq \frac{1}{3^{n(d+2)}}\left[J_{m}+(3^{n(d+2)}-3^{d+1})E_{m}^{2}+3^{d+1}J_{m}\right]\\
&\leq E_{m}^{2}+\frac{C}{3^{n(d+2)}}J_{m}. 
\end{align*}
\end{proof}}

{ Our next lemma shows that, but perturbing $\ell$, we can make $E$ and $E^*$ positive.}

\begin{lem}\label{lem:strictpos}
Let $\ell$ so that 
\begin{equation*}
E(\ell)=\lim_{n\rightarrow\infty}\EE[\mu(G_{n}, \om, \ell)]=\lim_{n\rightarrow\infty} \EE[\mu^{*}(G_{n}, \om, \ell)]=E^{*}(\ell).
\end{equation*}
There exists $\c[strictpos]=\c[strictpos](d, \la, \La)$ so that for any $\ga>0$, for any $n$, 
\begin{equation}
\EE[\mu(G_{n}, \om, \ell-\ga)]\geq \c[strictpos] \ga^{d+1}.
\end{equation}
Analogously, 
\begin{equation}
\EE[\mu^{*}(G_{n}, \om, -\ell+\ga)]=\EE[\mu(G_{n}, \om^{*}, \ell-\ga)]\geq \c[strictpos]\ga^{d+1}.
\end{equation}
\end{lem}
\begin{proof}
First, we observe that by Lemma \ref{both0}, $E(\ell)=0$. By the subadditive ergodic theorem, we choose $N=N(\delta)$ sufficiently large so that $\EE[\mu(G_{N}, \om, \ell)]\leq \delta$. 

Let $w$ solve 
\begin{equation*}
\begin{cases}
w_{t}+F(D^{2}w, x, t, \om)=\ell & \text{in}\quad G_{N},\\
w=0 & \text{on}\quad \partial_{p} G_{N}.
\end{cases}
\end{equation*}
 
 Since $w\in S(G_{N}, \om, \ell)$, by Lemma \ref{muabp},
 \begin{equation*}
 0\leq \inf_{G_{N}} w+\c[ptf]3^{2N} \mu(G_{N}, \om, \ell)^{1/(d+1)},
 \end{equation*}
 which implies that 
 \begin{equation}\label{oiy}
 \PP[w\leq -2^{1/(d+1)}\c[ptf]3^{2N} \delta^{1/(d+1)}]\leq \PP[\mu(G_{N}, \om, \ell)\geq 2\delta]\leq \frac{1}{2}.
 \end{equation}

Let $\tilde{w}:=w-C\ga(\frac{1}{2}|x|^{2}-3^{2N})+\frac{\ga}{2}(3^{2N}-t)$ for $C$ to be chosen. By \pref{oiy},
 \begin{equation*}
\PP[\tilde{w}\geq -2\c[ptf] 3^{2N}\delta^{1/(d+1)}+C\ga 3^{2N}]\geq \frac{1}{2}.
\end{equation*}
 
Next we consider that there exists $C=C(d, \la)$ so that $\tilde{w} \in S(G_{N}, \om, \ell-\ga)$. We verify that 
\begin{align*}
\tilde{w}_{t}+F(D^{2}\tilde{w}, x, t, \om)&=w_{t}-\frac{\ga}{2}+F(D^{2}w-C\ga Id, x, t, \om)\\
&\geq w_{t}-\frac{\ga}{2}+F(D^{2}w, x, t, \om)+\la|C\ga Id|\\
&=\ell-\frac{\ga}{2}+C\la \ga d\geq \ell-\ga
\end{align*}
for $C=C(\la, d)$.
Since $\tilde{w}\geq 0$ on $\partial_{p}G_{N}$, by Lemma \ref{muabp},
\begin{equation*}
\PP[\mu(G_{N}, \om, \ell-\ga)\geq C\ga^{d+1}-C\delta]\geq \frac{1}{2}.
\end{equation*}

Therefore, for all $n\leq N$, 
\begin{equation*}
\EE[\mu(G_{n}, \om, \ell+\ga)]\geq C(\ga^{d+1}-\delta).
\end{equation*}

Sending $\delta\rightarrow 0$, $N(\delta)\rightarrow\infty$ and we have the claim by letting $\c[strictpos]=C$. 
\end{proof}

We are now ready to obtain a rate of decay on the second moments of $\mu$.

\begin{thm}\label{thmmom}
There exists $\tau=\tau(\la, \La, d)\in (0, 1)$ and $\c[mom]=\c[mom](\la, \la, d)$ such that for all $m\in \mathbb{N}$, for each $M\in\mathbb{S}^{d}$, 
\begin{equation}\label{momrate}
J_{m}(\overline{F}(M),M)+J^{*}_{m}(-\overline{F}(M),M)\leq \c[mom](1+|M|)^{2(d+1)}K_{0}^{2(d+1)}\tau^{m}.
\end{equation}
\end{thm}

\begin{proof}
We fix $M\in \mathbb{S}^{d}$ and drop the dependence on $\overline{F}(M)$ (although we mention where it is used). In order to prove \pref{momrate}, it is enough to prove that there exists an increasing sequence of integers $\left\{m_{k}\right\}$ so that $|m_{k+1}-m_{k}|\leq C=C(d, \la, \La)$ with 
\begin{equation}\label{gridmom}
J_{m_{k}}(-3^{-k})+J^{*}_{m_{k}}(3^{-k})\leq C(1+|M|)^{2(d+1)}K_{0}^{2(d+1)} 3^{-2k(d+1)}.
\end{equation}
Recall that $|\overline{F}(M)|\leq CK_{0}^{d+1}(1+|M|)^{d+1}$. By \pref{bndest} and scaling, it is enough to assume that we work with 
\begin{equation*}
J_{k}:= \frac{J_{k}}{C(1+|M|)^{2(d+1)}K_{0}^{2(d+1)}}
\end{equation*}
so that $|J_{k}|\leq 1$, and then to prove
\begin{equation}\label{semirate}
J_{m_{k}}(-3^{-k})+J^{*}_{m_{k}}(3^{-k})\leq C3^{-2k(d+1)}.
\end{equation}
Let $m_{0}=0$. Suppose that \pref{semirate} holds for the level $m_{k-1}$. We would like to find $m_{k}$ satisfying \pref{semirate} such that $m_{k}-m_{k-1}\leq C$. We aim to set up Lemma \ref{decaylem}, and then choose $\ga=3^{-k}$. Given $n_{0}, \eta_{0}$ as in Lemma \ref{decaylem}, we seek $m$ satisfying \pref{vardecay}. 

{ Consider that by Lemma \ref{lem:iid}, }
\begin{equation}\label{decayprod}
{ J_{m-n_{0}}(-3^{-k})\leq E^{2}_{m-n_{1}}(-3^{-k})+\frac{\c[mixdecay]}{3^{(n_{1}-n_{0})(d+2)}}J_{m-n_{1}}(-3^{-k}).}
\end{equation}

If we can find a choice of $m$ so that for a fixed $n_{1}$, $\eta_{1}$, 
\begin{equation}\label{close1}
 E_{m-n_{1}}(-3^{-k})\leq (1+\eta_{1})^{1/2}E_{m}(-3^{-k})\quad,\quad E^{*}_{m-n_{1}}(3^{-k})\leq (1+\eta_{1})^{1/2}E^{*}_{m}(3^{-k}),
 \end{equation}
 and
 \begin{equation}\label{close2}
J_{m-n_{1}}(-3^{-k})\leq (1+\eta_{1})J_{m}(-3^{-k})\quad,\quad J^{*}_{m-n_{1}}(3^{-k})\leq (1+\eta_{1})J^{*}_{m}(3^{-k}),
 \end{equation}
then substituting this into \pref{decayprod}, 
{
 \begin{align*}
J_{m-n_{0}}(-3^{-k})&\leq (1+\eta_{1})\left[E^{2}_{m}(-3^{-k})+\frac{\c[mixdecay]}{3^{(n_{1}-n_{0})(d+2)}}J_{m}(-3^{-k})\right]\\
&\leq (1+\eta_{1})\left[E_{m}^{2}(-3^{-k})+\frac{\c[mixdecay]}{3^{(n_{1}-n_{0})(d+2)}}J_{m-n_{0}}(-3^{-k})\right],
\end{align*}}
which implies that {
\begin{equation*}
\left[1-(1+\eta_{1})\frac{\c[mixdecay]}{3^{(n_{1}-n_{0})(d+2)}}\right] J_{m-n_{0}}(-3^{-k})\leq  (1+\eta_{1})E^{2}_{m}(-3^{-k}).
\end{equation*}}
Similarly, by \pref{vardecay*}, {
\begin{equation*}
\left[1-(1+\eta_{1})\frac{\c[mixdecay]}{3^{(n_{1}-n_{0})(d+2)}}\right] J^{*}_{m-n_{0}}(3^{-k})\leq  (1+\eta_{1})E^{*2}_{m}(3^{-k}).
\end{equation*}}
Choosing $n_{1}(d, \la, \La), \eta_{1}(d, \la, \La)$ so that  {
\begin{equation}\label{nchoice}
\left[1-(1+\eta_{1})\frac{\c[mixdecay]}{3^{(n_{1}-n_{0})(d+2)}}\right]^{-1}(1+\eta_{1})\leq 1+\eta_{0},
\end{equation}}
we may apply Lemma \ref{decaylem}, to conclude that for $m$ satisfying \pref{close1} and \pref{close2},
\begin{equation*}
J_{m}(-3^{-k})+J^{*}_{m}(3^{-k})\leq C3^{-2k(d+1)}.
\end{equation*}

The problem reduces to finding a choice of $m$ satisfying \pref{close1} and \pref{close2}, such that $m$ is a bounded distance away from $m_{k-1}$. This is where we will use the inductive hypothesis. We claim that for given $n_{1}, \eta_{1}$, there exists $m$ such that \pref{close1} and \pref{close2} hold, and 
\begin{equation}
n_{1}\leq m\leq m_{k-1}+C\log \left[C\left(J_{m_{k-1}}(-3^{-(k-1)})+J^{*}_{m_{k-1}}(3^{-(k-1)})\right)\right].
\end{equation}

Consider that for all $m$, by Lemma \ref{lem:strictpos}, since we are solving with right hand side $\overline{F}(M)$ (and here is the only place where we use that the right hand side is $\overline{F}(M))$, 
\begin{equation*}
\c[strictpos]3^{-(k-1)(d+1)}\leq E_{m}(-3^{-(k-1)})\quad\text{and}\quad \c[strictpos]3^{-(k-1)(d+1)}\leq E^{*}_{m}(3^{-(k-1)}).
\end{equation*}

This implies that for any $N$, 
\begin{align*}
&\prod_{j=1}^{N}\frac{J_{m_{k-1}+(j-1)n_{1}}(-3^{-(k-1)})}{J_{m_{k-1}+jn_{1}}(-3^{-(k-1)})}\leq C \frac{J_{m_{k-1}}(-3^{-(k-1)})}{3^{-2(k-1)(d+1)}},\\
&\prod_{j=1}^{N}\frac{J^{*}_{m_{k-1}+(j-1)n_{1}}(3^{-(k-1)})}{J^{*}_{m_{k-1}+jn_{1}}(3^{-(k-1)})}\leq C \frac{J^{*}_{m_{k-1}}(3^{-(k-1)})}{3^{-2(k-1)(d+1)}},\\
&\prod_{j=1}^{N}\frac{E_{m_{k-1}+(j-1)n_{1}}(-3^{-(k-1)})}{E_{m_{k-1}+jn_{1}}(-3^{-(k-1)})}\leq C \frac{E_{m_{k-1}}(-3^{-(k-1)})}{3^{-(k-1)(d+1)}},\\
&\prod_{j=1}^{N}\frac{E^{*}_{m_{k-1}+(j-1)n_{1}}(3^{-(k-1)})}{E^{*}_{m_{k-1}+jn_{1}}(3^{-(k-1)})}\leq C \frac{E^{*}_{m_{k-1}}(3^{-(k-1)})}{3^{-(k-1)(d+1)}}.
\end{align*}

Since each individual term in the product is bounded from below by 1, this implies that there exists some element $j^{i}$ for $i=1, 2, 3, 4$ such that
\begin{align*}
&\frac{J_{m_{k-1}+(j^{1}-1)n_{1}}(-3^{-(k-1)})}{J_{m_{k-1}+j^{1}n_{1}}(-3^{-(k-1)})}\leq C \left(\frac{J_{m_{k-1}}(-3^{-(k-1)})}{3^{-2(k-1)(d+1)}}\right)^{1/N},\\
&\frac{J^{*}_{m_{k-1}+(j^{2}-1)n_{1}}(3^{-(k-1)})}{J^{*}_{m_{k-1}+j^{2}n_{1}}(3^{-(k-1)})}\leq  C \left(\frac{J^{*}_{m_{k-1}}(3^{-(k-1)})}{3^{-2(k-1)(d+1)}}\right)^{1/N},\\
&\frac{E_{m_{k-1}+(j^{3}-1)n_{1}}(-3^{-(k-1)})}{E_{m_{k-1}+j^{3}n_{1}}(-3^{-(k-1)})}\leq C \left(\frac{J_{m_{k-1}}(-3^{-(k-1)})}{3^{-2(k-1)(d+1)}}\right)^{1/2N},\\
&\frac{E^{*}_{m_{k-1}+(j^{4}-1)n_{1}}(3^{-(k-1)})}{E^{*}_{m_{k-1}+j^{4}n_{1}}(3^{-(k-1)})}\leq  C \left(\frac{J^{*}_{m_{k-1}}(3^{-(k-1)})}{3^{-2(k-1)(d+1)}}\right)^{1/2N}.
\end{align*}

Let
\begin{equation*}
N:=\ceil[\Bigg]{C\frac{\log [3^{2(k-1)(d+1)}(J_{m_{k-1}}(-3^{-(k-1)})+J^{*}_{m_{k-1}}(3^{k-1}))]}{\log (1+\delta_{1})}},
\end{equation*}
and set $m:=m_{k-1}+jn_{1}$ for $j:=\max_{i}\left\{j^{i}\right\}\leq N$. Applying the monotonicity, this choice of $m$ satisfies \pref{close1} and \pref{close2}. Define $m_{k}:=m$, and this implies by the inductive hypothesis that
\begin{align*}
m_{k}&\leq m_{k-1}+C\log [3^{2(k-1)(d+1)}(J_{m_{k-1}}(-3^{-(k-1)})+J^{*}_{m_{k-1}}(3^{k-1}))]\\
&\leq m_{k-1}+C\log [C3^{2(k-1)(d+1)}3^{-2(k-1)(d+1)}]\leq m_{k-1}+C.
\end{align*}
This completes the induction, and the proof of \pref{gridmom}. By the monotonicity in the right hand side $\ell$, this actually yields a sequence $\left\{m_{k}\right\}$ so that $|m_{k}-m_{k-1}|\leq C$ for all $k$, and 
\begin{equation*}
J_{m_{k}}+J^{*}_{m_{k}}\leq C3^{-2k(d+1)}.
\end{equation*}
 Using the monotonicity of $J_{m}$ in $m$ to interpolate between points $m=3^{m_{k}}$, we obtain \pref{momrate} for some $\c[mom]$.
\end{proof}

Using this rate on the decay of the second moments, we apply Chebyshev's inequality to obtain a rate on the decay of $\mu$. 

\begin{cor}\label{cormup}
{For every $p<d+2$, there exists $c=c(p, \la, \La, d)$ and $\al=\al(\la, \La, p, d)$ so that for all $m\in \NN$, for all $\nu\geq 1$, }
\begin{equation}\label{muprob}{
\PP[\mu(G_{m}, \om, \overline{F}(M), M)\geq (1+|M|)^{d+1}K_{0}^{d+1}3^{-m\al}\nu]\leq \exp(-c\nu3^{mp}),}
\end{equation}
and
\begin{equation}\label{mu*prob}{
\PP[\mu^{*}(G_{m}, \om, \overline{F}(M), M)\geq (1+|M|)^{d+1}K_{0}^{d+1}3^{-m\al}\nu]\leq \exp(-c\nu3^{mp}).}
\end{equation}
\end{cor}

\begin{proof}
We only prove \pref{muprob}, since \pref{mu*prob} follows by identical arguments. Without loss of generality, we assume that $M=0$, and we drop the dependence on $\overline{F}(0)$. 

Fix $m\in \NN$ and let $n\in \NN$ to be chosen. We consider decomposing $G_{m+n+1}=\bigcup_{i=1}^{3^{d+2}}\mathcal{G}_{n}^{i}$ where $\mathcal{G}_{n}^{i}=\bigcup_{j=1}^{3^{m(d+2)}} G_{n}^{ij}$ is a collection of subcubes of size $G_{n}$ such that each of the subcubes of size $G_{n}$ is separated by distance at least 1.

By the finite range of dependence assumption \pref{f1'}, for each $i$, 
\begin{equation}
\mu(G_{n}^{ij}, \om)\quad\text{and}\quad \mu(G_{n}^{ik}, \om)\quad\text{are independent if $j\neq k$}.
\end{equation}

Using this decomposition yields that 
\begin{align*}
\log\EE[\exp(\nu 3^{m(d+2)}&\mu(G_{m+n+1},\om))]\leq \log\EE\left[\prod_{i=1}^{3^{d+2}}\prod_{j=1}^{3^{m(d+2)}}\exp\left(\nu 3^{-(d+2)}\mu(G_{n}^{ij}, \om)\right)\right]\\
&\leq 3^{-(d+2)}\sum_{i=1}^{3^{(d+2)}} \log\EE\left[\prod_{j=1}^{3^{m(d+2)}}\exp\left(\nu \mu(G_{n}^{ij}, \om)\right)\right]\\
&=3^{-(d+2)}\sum_{i=1}^{3^{(d+2)}}\log\left(\prod_{j=1}^{3^{m(d+2)}} \EE\left[\exp\left(\nu\mu(G_{n}^{ij}, \om)\right)\right]\right)\\
&=3^{m(d+2)} \log\EE[\exp(\nu \mu(G_{n}, \om))],
\end{align*}
where the last line holds by stationarity. Moreover, if we choose $\nu=CK_{0}^{-1/(d+1)}$, then $\nu \mu(G_{n}, \om)\leq 1$ almost surely. Using the elementary inequalities
\begin{equation*}
\begin{cases}
\exp(s)\leq 1+2s &\text{for all}\quad0\leq s\leq 1,\\
\log(1+s)\leq s &\text{for all}\quad s\geq 0,
\end{cases}
\end{equation*} 
yields that for this choice of $\nu$, 
\begin{align}
\log\EE[\exp(C K_{0}^{-(d+1)}3^{m(d+2)}\mu(G_{m+n+1},\om))]&\leq 3^{m(d+2)}\EE[CK_{0}^{-(d+1)}\mu(G_{n}, \om)]\notag\\
&\leq C3^{m(d+2)}\tau^{n}\label{wahoo}
\end{align}
by Theorem \ref{thmmom}. 

Therefore, by Chebyshev's inequality and \pref{wahoo}, this yields that 
\begin{align*}
&\PP\left[\mu(G_{m+n+1}, \om)\geq K_{0}^{d+1}\nu\right]\\
&\leq \PP\left[\exp(K_{0}^{-(d+1)}3^{m(d+2)}\mu(G_{m+n+1}, \om))\geq \exp(3^{m(d+2)}\nu)\right]\\
&\leq C\exp(-3^{m(d+2)}(\nu-\tau^{n})).
\end{align*}

Letting $\nu=\frac{\tau^{n}\nu}{2}$, and using that $\nu\geq 1$, we have that 
\begin{equation*}
\PP\left[\mu(G_{m+n+1}, \om)\geq C\tau^{n}K_{0}^{d+1}\nu\right]\leq C\exp(-3^{m(d+2)}\tau^{n}\nu).
\end{equation*}

{Choosing $n\sim \left\lfloor \frac{mp\log 3}{2(p\log 3+|\log \tau|)}\right\rfloor\leq \frac{m}{2}$ implies that $c3^{-mp}\leq \tau^{n}\leq C3^{-mp}$, which yields that 
\begin{equation*}
\PP\left[\mu(G_{m+n+1}, \om)\geq C3^{-mp}K_{0}^{d+1}\nu\right]\leq C\exp(-3^{m(d+2-p)}\nu).
\end{equation*}}
{
Relabeling $m=m+n+1$ and $p=d+2-p$ yields that there exists $\al=\al(\la, \La, p, d)$ such that
\begin{equation*}
\PP\left[\mu(G_{m}, \om)\geq C3^{-m\al}K_{0}^{d+1}\nu\right]\leq C\exp(-3^{mp}\nu).
\end{equation*}}
\end{proof}

\section{The Proof of Theorem \ref{mainthm}}\label{sec:qptm}
We finally present the rate for homogenization in probability using Theorem \ref{thmmom}. This follows a general procedure which has been shown in \cite{cs, asellip, linhomog}. However, for completeness we provide the argument here as well, {similar to the approach of \cite{asellip}}. As mentioned in \cite{cs, asellip, linhomog},  if the limiting function $u$ is $C^{2}(\RR^{d+1})$(i.e. $C^{2}(\RR^{d})\cap C^{1}([0,T])$), then obtaining a rate for the homogenization is straightforward. Studying $\lim_{\ve\rightarrow 0} w^{\ve}$ where $w^{\ve}$ solves \pref{corrector} is equivalent to the stochastic homogenization of \pref{homeq} when the limiting function is of the form $u(x,t)=bt+\frac{1}{2}x\cdot Mx$. By \pref{quaddecay} and Chebyshev's inequality, a rate on the decay of $\mu(G_{1/\ve}, \om)$ immediately yields a rate in probability for the decay of $w^{\ve}$. If $u\in C^{2}$, then by replacing $u$ with its second-order Taylor series expansion with cubic error, we obtain a rate for $u^{\ve}-u$. In general, since $u$ is not necessarily $C^{2}$, we must argue that one can still approximate $u$ by a quadratic expansion. This type of approximation is the motivation for the theory of $\delta$-viscosity solutions, which was introduced in the elliptic setting in \cite{cs}, and generalized to the parabolic setting by Turanova \cite{olga1}. The rate in \cite{linhomog} was obtained by using this regularization procedure.  

For clarity and for a more general approach, we choose to present the argument in terms of a quantified comparison principle as in \cite{asellip}. We revert to quantifying the traditional ``doubling variables" arguments used in the theory of viscosity solutions (see for example \cite{users, crandall}). We are informed that this is related to a forthcoming work by Armstrong and Daniel \cite{scottjp}, who generalize this method to finite difference schemes for fully nonlinear uniformly parabolic equations. The next series of results are entirely deterministic, and therefore we suppress the dependence on the random parameter $\om$. 

We first present a result relating the measure of the parabolic subdifferential with the measure of the corresponding touching points in physical space-time. 
\begin{prop}\label{qviscosity}
Let $u,v$ such that  
\begin{equation}\label{eqvisc}
u_{t}+\MM^{-}(D^{2}u)-R_{0}\leq 0\leq v_{t}+\MM^{+}(D^{2}v)+R_{0}\quad\text{in}\quad U_{T}.
\end{equation}
Assume $\delta>0$, and let $V=\overline{V}\subseteq U_{T}\times U_{T}$ and $W\subseteq \RR^{d+1}\times \RR^{d+1}$, such that for all $((p,h), (q, k))\in W$, 
\begin{align*}
\left\{(x,t,y,s): \right.&\left.\sup_{U_{T}\times U_{T}: \tau\leq t, \sigma \leq s } u(\xi,\tau)-v(\eta, \sigma)-\frac{1}{2\delta}\left[|\xi-\eta|^{2}+(\tau-\sigma)^{2}\right]-p\cdot \xi\right.\\
&-q\cdot \eta=u(x,t)-v(y,s)-\frac{1}{2\delta}\left[|x-y|^{2}+(t-s)^{2}\right]-p\cdot x-q\cdot y,\\
&h=u(x,t)-\frac{1}{2\delta}[|x-y|^{2}+(t-s)^{2}]-p\cdot x\\
&\left.k=-v(y,s)-\frac{1}{2\delta}[|x-y|^{2}+(t-s)^{2}]-q\cdot y\right\}\subseteq V.
\end{align*}

Then there exists a constant $C=C(\la, \La, d, U_{T})$ such that 
\begin{equation}\label{measbound}
|W|\leq C\left(R_{0}+\delta^{-1}\right)^{2d+2}|V|.
\end{equation}
\end{prop}

\begin{proof}
Without loss of generality, we may assume by scaling that $U_{T}\subseteq Q_{1}(0, 1)$. As usual, we constantly relabel $C$ for a constant which only depends on $\la, \La, d$. For $i=1,2$, let $(x_{i}, t_{i}, y_{i}, s_{i}, p_{i}, q_{i}, h_{i}, k_{i})$ satisfy
\begin{align*}
&\sup_{U_{T}\times U_{T}, \tau\leq t_{i}, \sigma\leq s_{i}} u(x,\tau)-v(y,\sigma)-\frac{1}{2\delta}\left(|x-y|^{2}+(\tau-\sigma)^{2}\right)-p_{i}\cdot x-q_{i}\cdot y\\
&=u(x_{i}, t_{i})-v(y_{i}, s_{i})-\frac{1}{2\delta}\left(|x_{i}-y_{i}|^{2}+(t_{i}-s_{i})^{2}\right)-p_{i}\cdot x_{i}-q_{i}\cdot y_{i}=h_{i}+k_{i},
\end{align*}
and let 
\begin{equation}
\De= (|x_{1}-x_{2}|^{2}+|y_{1}-y_{2}|^{2}+|t_{1}-t_{2}|+|s_{1}-s_{2}|)^{1/2}.
\end{equation}
We claim that  
\begin{equation}\label{subgoal}
(|p_{1}-p_{2}|^{2}+|q_{1}-q_{2}|^{2}+|h_{1}-h_{2}|^{2}+|k_{1}-k_{2}|^{2})^{1/2}\leq C (1+\delta^{-1})\De+o(\De). 
\end{equation}
as $|\De|\rightarrow 0$.

If \pref{subgoal} holds, then one can obtain \pref{measbound} using standard measure-theoretic arguments. A priori, this may not be apparent since the left hand side of \pref{subgoal} corresponds to the Euclidean distance between points in $\RR^{d+1}$, whereas $\De$ corresponds to the parabolic distance under the metric $d[\cdot, \cdot]$.  { However, the parabolic cylinders have the appropriate doubling property with respect to Lebesgue measure, and thus standard measure-theoretic arguments apply.}

We prove a series of claims, using standard techniques in the method of doubling variables. 
\begin{claim}
For each $i$, 
\begin{equation}\label{tbnd}
|t_{i}-s_{i}|\leq \delta R_{0}+C.
\end{equation}
\end{claim}
Consider that the map 
\begin{equation*}
(x,t)\rightarrow u(x,t)-\frac{1}{2\delta}[|x-y_{1}|^{2}+(t-s_{1})^{2}]-p_{1}\cdot x
\end{equation*}
achieves its maximum over $U\times (0, t_{1}]$ at $(x_{1}, t_{1})$. Therefore, by \pref{eqvisc}, 
\begin{equation*}
\frac{1}{\delta}(t_{1}-s_{1})+\MM^{-}\left(\delta^{-1} Id\right)\leq R_{0},
\end{equation*}
implying that 
\begin{equation}\label{1sidet}
t_{1}-s_{1}\leq \delta[R_{0}-(-C\delta^{-1})]=\delta R_{0}+C.
\end{equation}

Similarly, the map
\begin{equation*}
(y,s)\rightarrow v(y,s)+\frac{1}{2\delta}[|x_{1}-y|^{2}+(t_{1}-s)^{2}]+q_{1}\cdot y
\end{equation*}
achieves its minimum over $U\times (0, s_{1}]$ at $(y_{1}, s_{1})$. By \pref{eqvisc}, 
\begin{equation}\label{2sidet}
t_{1}-s_{1}\geq \delta(-R_{0}-C\delta^{-1})= -\delta R_{0}-C.
\end{equation}
Combining \pref{1sidet} and \pref{2sidet} yields \pref{tbnd}. 

\begin{claim}
Let $u_{t}+\MM^{+}(D^{2}u)\geq -1$ in $Q_{1}$. Let $(p_{1}, h_{1})\in \mathcal{P}((x_{1},t_{1}); u)$ and $(p_{2}, h_{2})\in \mathcal{P}((x_{2}, t_{2}); u)$. Then
\begin{equation}\label{lip}
|p_{1}-p_{2}|^{2}+|h_{1}-h_{2}|^{2}\leq C\left(|x_{1}-x_{2}|^{2}+|t_{1}-t_{2}|^{2}+|x_{1}-x_{2}|^{4}+|t_{1}-t_{2}|\right).
\end{equation}
\end{claim}
Without loss of generality, by subtracting a plane and translating, we may assume that $(p_{2}, h_{2})=(0,0)$ and $(x_{2}, t_{2})=(0,0)$. The claim will follow from the regularity of $\Ga^{u}$ (Lemma \ref{mereg}). Since $(x_{1}, t_{1}), (0,0)\in \left\{u=\Ga^{u}\right\}$, and $D\Ga^{u}$ is Lipschitz continuous, this implies that 
\begin{equation*}
|p_{1}|\leq C(|x_{1}|^{2}+|t_{1}|)^{1/2}.
\end{equation*}
To estimate $|h_{1}|$, we again apply the regularity of $\Ga^{u}$ and the bound on $|p_{1}|$ to conclude that 
\begin{align*}
|h_{1}|=|h_{1}-h_{2}|=|u(x_{1}, t_{1})-p_{1}\cdot x_{1}-u(x_{2}, t_{2})|\leq C(|x_{1}|^{2}+|t_{1}|)^{1/2}(1+|x_{1}|).
\end{align*}
Therefore, 
\begin{equation*}
|h_{1}|^{2}\leq C^{2}(|x_{1}|^{2}+|t_{1}|)(1+|x_{1}|)^{2}\leq C(|x_{1}|^{2}+|t_{1}|^{2}+|x_{1}|^{4}+|t_{1}|).
\end{equation*}
Combining these observations yields \pref{lip}. 

Next, we apply these observations to the parabolic subdifferentials. For simplicity, we adopt some notation. Without loss of generality, assume that $s_{1}\geq s_{2}$. Let $T_{min}:=\min\left\{t_{1}, t_{2}, s_{2}\right\}$ and $T_{max}:=\max \left\{t_{1}, t_{2}, s_{1}\right\}.$ Notice that by \pref{tbnd}, $T_{max}-T_{min}\leq \delta R_{0}+C+\De^{2}:=\ga^{2}$. Therefore, $(x_{1}, t_{1}), (x_{2}, t_{2})\in Q_{\ga}(x_{1}, T_{max})$. Let 
\begin{equation*}
\tilde{u}(x,t):=-u(x,t)+\frac{1}{2\delta}[|x-y_{1}|^{2}+(t-s_{1})^{2}].
\end{equation*}
This implies that 
\begin{align}\label{tilueq}
\tilde{u}_{t}+\MM^{+}(D^{2}\tilde{u})&=-u_{t}+\delta^{-1}(t-s_{1})+\MM^{+}(-D^{2}u+\delta^{-1}Id)\\
&\geq -u_{t}+\delta^{-1}(t-s_{1})-\MM^{-}(D^{2}u)-\delta^{-1}C\notag\\
&\geq -R_{0}-C(1+\delta R_{0}+\De^{2})\delta^{-1}\notag\\
&\geq -C(R_{0}+\delta^{-1}(1+\De^{2}))\quad\text{in}\quad Q_{\ga}(x_{1}, T_{M}).\notag
\end{align}

We next find elements in the parabolic subdifferential of $\tilde{u}$. 
\begin{claim}
\begin{equation}\label{ezsd}
(-p_{1}, \tilde{u}(x_{1}, t_{1})+p_{1}\cdot x_{1})\in \mathcal{P}((x_{1}, t_{1}); \tilde{u}).
\end{equation}
\end{claim}
Since
\begin{align*}
u(x_{1}, t_{1})-\frac{1}{2\delta}\left[|x_{1}-y_{1}|^{2}+(t_{1}-s_{1})^{2}\right]&-p_{1}\cdot x_{1}\\
&\geq u(x,t)-\frac{1}{2\delta}[|x-y_{1}|^{2}+(t-s_{1})^{2}]-p_{1}\cdot x
\end{align*}
for all $t\leq t_{1}$, $x\in U$, this implies that 
\begin{align*}
\tilde{u}(x_{1}, t_{1})-(-p_{1}\cdot x_{1})&=-u(x_{1}, t_{1})+\frac{1}{2\delta}(|x_{1}-y_{1}|^{2}+(t_{1}-s_{1})^{2})+p_{1}\cdot x_{1}\\
&\leq \tilde{u}(x,t)-(-p_{1}\cdot x)
\end{align*}
for all $t\leq t_{1}$, $x\in U$. This yields \pref{ezsd}. 

\begin{claim}
\begin{equation}\label{hardsd}
\left(-p_{2}+\frac{y_{2}-y_{1}}{\delta}, \tilde{u}(x_{2}, t_{2})+\left(p_{2}-\frac{y_{2}-y_{1}}{\delta}\right)\cdot x_{2}\right)\in \mathcal{P}((x_{2}, t_{2}); \tilde{u}).
\end{equation}
\end{claim}
Since 
\begin{align*}
&-u(x,t)+\frac{1}{2\delta}\left[|x-y_{2}|^{2}+(t-s_{2})^{2}\right]+p_{2}\cdot x\\
&=\tilde{u}(x,t)+\frac{1}{2\delta}\left[|x-y_{2}|^{2}+(t-s_{2})^{2}-|x-y_{1}|^{2}-(t-s_{1})^{2}\right]+p_{2}\cdot x\\
&=\tilde{u}(x,t)+\left(\frac{1}{\delta}(-y_{2}+y_{1})+p_{2}\right)\cdot x+\frac{1}{2\delta}\left[(t-s_{2})^{2}-(t-s_{1})^{2}+|y_{2}^{2}|-|y_{1}|^{2}\right],
\end{align*}
we obtain that 
\begin{align*}
&\tilde{u}(x_{2}, t_{2})+\left(\frac{1}{\delta}(-y_{2}+y_{1})+p_{2}\right)\cdot x_{2}+\frac{1}{2\delta}\left[(t_{2}-s_{2})^{2}-(t_{2}-s_{1})^{2}\right]\\
&\leq \tilde{u}(x,t)+\left(\frac{1}{\delta}(-y_{2}+y_{1})+p_{2}\right)\cdot x+\frac{1}{2\delta}\left[(t-s_{2})^{2}-(t-s_{1})^{2}\right].
\end{align*}
Simplifying yields that 
\begin{align*}
\tilde{u}(x_{2}, t_{2})+\left(\frac{1}{\delta}(-y_{2}+y_{1})+p_{2}\right)\cdot x_{2}+\frac{1}{\delta}&\left[-(t_{2}-t)(s_{2}-s_{1})\right]\\
&\leq \tilde{u}(x,t)+\left(\frac{1}{\delta}(-y_{2}+y_{1})+p_{2}\right)\cdot x.
\end{align*}
Therefore, for $t\leq t_{2}$, since $s_{1}\geq s_{2}$, 
\begin{equation*}
\tilde{u}(x_{2}, t_{2})+\left(\frac{1}{\delta}(-y_{2}+y_{1})+p_{2}\right)\cdot x_{2} \leq \tilde{u}(x,t)+\left(\frac{1}{\delta}(-y_{2}+y_{1})+p_{2}\right)\cdot x,
\end{equation*}
which yields the claim. 

By combining \pref{lip}, \pref{tilueq}, \pref{ezsd}, and \pref{hardsd}, 
\begin{align*}
&\left|p_{1}-p_{2}+\frac{1}{\delta}(y_{2}-y_{1})\right|^{2}\\
&+\left|\tilde{u}(x_{1}, t_{1})+p_{1}\cdot x_{1}-\tilde{u}(x_{2}, t_{2})-\left(p_{2}-\frac{1}{\delta}(y_{2}-y_{1})\right)\cdot x_{2}\right|^{2}\\
&\leq C\left[R_{0}+\delta^{-1}(1+\De^{2})\right]^{2}(|x_{1}-x_{2}|^{2}+|t_{1}-t_{2}|^{2}+|x_{1}-x_{2}|^{4}+|t_{1}-t_{2}|).
\end{align*}
Recall that 
\begin{equation*}
-\tilde{u}(x_{1}, t_{1})-p_{1}\cdot x_{1}=h_{1}
\end{equation*}
and 
\begin{align*}
&-\tilde{u}(x_{2}, t_{2})-\left(p_{2}-\frac{1}{\delta}(y_{2}-y_{1})\right)\cdot x_{2}\\
&=h_{2}+\frac{1}{2\delta}(|y_{2}|^{2}-|y_{1}|^{2})+\frac{1}{2\delta}\left[(t_{2}-s_{2})^{2}-(t_{2}-s_{1})^{2}\right]\\
&=h_{2}+\frac{1}{2\delta}\left[|y_{2}|^{2}-|y_{1}|^{2}+s_{2}^{2}-s_{1}^{2}-2t_{2}(s_{2}-s_{1})\right].
\end{align*}
Collecting terms yields that 
\begin{align*}
&|p_{1}-p_{2}|^{2}+|h_{1}-h_{2}|^{2}\\
&\leq C[R_{0}+\delta^{-1}(1+\De^{2})]^{2}\left[|x_{1}-x_{2}|^{2}+|t_{1}-t_{2}|^{2}+|x_{1}-x_{2}|^{4}+|t_{1}-t_{2}|\right]\\
&+\frac{1}{\delta^{2}}|y_{2}-y_{1}|^{2}+\frac{1}{4\delta^{2}}\left[|y_{2}|^{2}-|y_{1}|^{2}+s_{2}^{2}-s_{1}^{2}-2t_{2}(s_{2}-s_{1})\right]^{2}\\
&\leq C[R_{0}+\delta^{-1}(1+\De^{2})]^{2}\De^{2}+\frac{1}{\delta^{2}}o(\De^{2})\\
&\leq C[R_{0}+\delta^{-1}]^{2}\De^{2}+o(\De^{2}),
\end{align*}
which implies that 
\begin{equation*}
\left(|p_{1}-p_{2}|^{2}+|h_{1}-h_{2}|^{2}\right)^{1/2}\leq C(R_{0}+\delta^{-1})\De+o(\De).\end{equation*}
An analogous argument yields that
\begin{equation*}
\left(|q_{1}-q_{2}|^{2}+|k_{1}-k_{2}|^{2}\right)^{1/2}\leq C(R_{0}+\delta^{-1})\De+o(\De).\end{equation*}
Combined, this yields \pref{subgoal}.
\end{proof}

Next, we show that if $|u-u^{\ve}|$ is large somewhere, then we can find a matrix $M^{*}$ and a parabolic cube $G^{*}$ so that $\mu(G^{*}, \overline{F}(M^{*}), M^{*})$ is very large. We mention that both $M^{*}$ and $G^{*}$ come from a countable family of matrices and cubes. In order to select $M^{*}, G^{*}$, we must construct the appropriate approximation of $u$ to argue that $u$ is close to a quadratic expansion. We will employ the $W^{3, \al}$ estimate proven in \cite{jp_par}, which yields an estimate on the measure of points which can be well-approximated by a quadratic expansion. We state the result slightly differently than it appears in \cite{jp_par}, in order to readily apply it for our purposes.

\begin{thm}[Theorem 1.2, \cite{jp_par}]\label{wpar}
Let $u_{t}+F(D^{2}u)=0$ in $Q_{1}$, $u=g$ on $\partial_{p}Q_{1}$, with $F$ uniformly parabolic. Let $Q\subseteq Q_{1}$. For each $\ka>0$, let
\begin{align*}
\Sigma_{\ka}:=&\left\{ (x,t)\in Q_{1}: \exists (M, \xi, b)\in \mathbb{S}^{d}\times \RR^{d}\times \RR, s.t.\,|M|\leq \ka, \forall (y,s)\in Q_{1}, s\leq t\right.\\
&\left.\left|u(y,s)-u(x,t)-b(s-t)-\xi\cdot (y-x)-\frac{1}{2}(y-x)\cdot M (y-x)\right|\right.\\
&\leq \left.\frac{1}{6}\ka\left(|x-y|^{3}+|s-t|^{3/2}\right)\right\}.
\end{align*}
There exists $C=C(\la, \La, d), \al=\al(\la, \La, d)$ so that for every $\ka>0$, 
\begin{align*}
\left|Q_{1}\setminus (\Sigma_{\ka}\right.\cap&\left. Q_{1/2}(0, -1/4))\right|\\
&\leq C\left(\frac{\ka}{\sup_{Q_{1}}\left[|u|+|F(0, \cdot, \cdot)|\right]+\norm{g}_{C^{0,1}(\partial_{p}Q_{1})})}\right)^{-\al}.
\end{align*}
\end{thm}

We note that $\Sigma_{\ka}$ corresponds to the set of points which can be touched monotonically in time by a quadratic expansion with controllable error. Moreover, the points in $\Sigma_{\ka}$ are touched from above and below by polynomials. We are now ready to show the existence of $M^{*}, G^{*}$. For simplicity, we say that a function $\Phi: U_{T}\times U_{T}$ achieves a monotone maximum at $(x_{0}, t_{0}, y_{0}, s_{0})$ if $\Phi(x_{0}, t_{0}, y_{0}, s_{0})\geq \Phi(x, t, y,s)$ for all $x,y\in U$, for all $t\leq t_{0}, s\leq s_{0}$.

\begin{prop}\label{grid}
Let $u, v$ satisfy
\begin{equation*}
\begin{cases}
u_{t}+\overline{F}(D^{2}u)=f(x,t)=v_{t}+F(D^{2}v, x, t)& \text{in}\quad U_{T},\\
u=v=g(x,t) & \text{on}\quad \partial_{p}U_{T},
\end{cases}
\end{equation*} 
so that 
\begin{equation*}
\norm{\overline{F}(0)}_{L^{\infty}(U_{T})}+\sup\norm{F(0, \cdot, \cdot)}_{L^{\infty}(U_{T})}+\norm{g}_{C^{0,1}(\partial_{p}U_{T})}+\norm{f}_{C^{0,1}(U_{T})}\leq R_{0}<+\infty.
\end{equation*}
There exists an exponent $\sigma=\sigma(\la, \La, d)\in (0,1)$ and constants $c=c(\la, \La, d, U_{T})$, $C=C(\la, \La, d, U_{T})$ so that for any $l\leq \eta$, if
\begin{equation}\label{error}
A:=\sup_{U_{T}} (u-v)\geq CR_{0}\eta^{\sigma}>0,
\end{equation}
then there exists $M^{*}\in \mathbb{S}^{d}$, $(y^{*}, s^{*})\in U_{T}$, such that 
\begin{itemize}
\item $|M^{*}|\leq \eta^{\sigma-1}$,
\item $l^{-1}M^{*}, \eta^{-1}y^{*}, \eta^{-2}s^{*}$ have integer entries,
\item $\mu((y^{*}, s^{*})+\eta G_{0}, \overline{F}(M^{*}), M^{*})\geq cA^{d+1}$.
\end{itemize} 
where $\eta G_{0}=\left(-\frac{\eta}{2}, \frac{\eta}{2}\right]^{d}\times \left(-\eta^{2}, 0\right]$.
\end{prop}

\begin{proof}
As usual, $c, C$ will denote constants which depend on universal quantities, which will vary line by line. We first point out some simplifications which we take without loss of generality. We assume that $R_{0}=1$, and $U_{T}\subseteq Q_{1}(0,1)$, and appropriately renormalize.

Next, we claim that we may replace $v$ by $\tilde{v}$ solving 
\begin{equation}\label{lilroom}
\begin{cases}
\tilde{v}_{t}+F(D^{2}\tilde{v}, x, t)=f(x,t)+cA & \text{in}\quad U_{T},\\
\tilde{v}=v & \text{on}\quad \partial_{p}U_{T}.
\end{cases}
\end{equation}
The Alexandrov-Backelman-Pucci-Krylov-Tso estimate \cite{wangreg1, cyriluis} yields that 
\begin{equation*}
\tilde{v}-v\leq CA\quad\text{in}\quad U_{T},
\end{equation*}
so by adjusting the constant in \pref{error}, we may take the replacement at no cost. 

Finally, we point out that by the Krylov-Safonov estimates \cite{wangreg1, cyriluis}, $u, v$ are Holder continuous, and since $R_{0}\leq 1$, there exists $\al(\la, \La, d)\in (0,1)$ such that 
\begin{equation}
\norm{u}_{C^{0, \al}(\overline{U_{T}})}+\norm{v}_{C^{0, \al}(\overline{U_{T}})}\leq C.
\end{equation}
Without loss of generality, assume that $\al\leq \frac{1}{2}$. Since $u=v$ on $\partial_{p}U_{T}$, this implies that for all $(x,t), (y,s)\in U_{T}$, 
\begin{equation*}
|u(x,t)-v(y,s)|\leq C\left(d[(x,t), \partial_{p}U_{T}]^{\al}+d[(y,s), \partial_{p}U_{T}]^{\al}+d[(x,t), (y,s)]^{\al}\right).
\end{equation*}
Consider the function 
\begin{equation*}
\Phi(x,t,y,s,p,q)=u(x,t)-v(y,s)-\frac{1}{2\delta}\left[|x-y|^{2}+(t-s)^{2})\right]-p\cdot x-q\cdot y.
\end{equation*}

Suppose there exists a point $(x_{0}, t_{0})$ such that $u(x_{0}, t_{0})-v(x_{0}, t_{0})\geq \frac{3}{4}A$. This implies that 
\begin{equation*}
\Phi(x_{0}, t_{0}, x_{0}, t_{0}, 0, 0)\geq \frac{3}{4}A. 
\end{equation*}

Let 
\begin{equation*}
U_{T}(\rho):=\left\{(x,t)\in U_{T}\times U_{T}: d[(x,t), \partial_{p}U_{T}]\geq \rho\right\}.
\end{equation*}
Let $p,q\in B_{r}$, where we define $r:=\frac{1}{8}A$. We would like to show that $\Phi(\cdot, \cdot, \cdot, \cdot, p,q)$ achieves it monotone maximum in $U_{T}(\rho)\times U_{T}(\rho)$ for some choice of $\rho$.

We note that 
\begin{align*}
\Phi(x,t,y,s,p,q)&=u(x,t)-v(y,s)-\frac{1}{2\delta}\left[|x-y|^{2}+(t-s)^{2}\right]-p\cdot x-q\cdot y\\
&\leq C\left(d[(x,t), \partial_{p}U_{T})]^{\al}+d[(y,s), \partial_{p}U_{T})]^{\al}+d[(x,t), (y,s)]^{\al}\right)\\
&-\frac{1}{2\delta}\left[|x-y|^{2}+(t-s)^{2}\right]+2r.
\end{align*}
By Young's inequality, 
\begin{equation*}
|x-y|^{\al}=A^{(2-\al)/2}[A^{-(2-\al)/\al}|x-y|^{2}]^{\al/2}\leq \frac{1}{8C}A+CA^{-(2-\al)/\al}|x-y|^{2},
\end{equation*}
and 
\begin{equation*}
|t-s|^{\al/2}=A^{(4-\al)/4}\left[A^{-(4-\al)/\al}|t-s|^{2}\right]^{\al/4}\leq \frac{1}{8C}A+CA^{-(4-\al)/\al}(t-s)^{2}.
\end{equation*}
Assume $A\leq 1$. This implies that $A^{-(2-\al)/\al}\leq A^{-(4-\al)/\al}$. 
Therefore, 
\begin{align*}
\Phi(x,y,t,s, p,q)&\leq Cd[(x,t), \partial_{p}U_{T}]^{\al}+Cd[(y,s), \partial_{p}U_{T}]^{\al}+\frac{1}{4}A\\
&+\frac{1}{4}A+C\left(A^{-(4-\al)/\al}-\frac{1}{2\delta}\right)\left[|x-y|^{2}+(t-s)^{2}\right].
\end{align*}

By letting 
\begin{equation}\label{deltachoice}
\delta\leq \frac{1}{2}A^{(4-\al)/\al},
\end{equation}
we have that 
\begin{equation*}
\Phi(x,y,t,s,p,q)\leq Cd[(x,t), \partial_{p}U_{T}]^{\al}+C[d(y,s), \partial_{p}U_{T}]^{\al}+\frac{1}{2}A.\\
\end{equation*}
Therefore, letting $\rho:= CA^{1/\al}$ yields that for any $p,q\in B_{r}$, $\Phi$ achieves its monotone maximum in $U_{T}(\rho)\times U_{T}(\rho)$. 

Using the language of Proposition \ref{qviscosity}, we let $W\subseteq\RR^{d+1}$ such that 
$Q_{r}\times Q_{r}\subseteq W$. This yields that 
\begin{align*}
V:=&\left\{(x,t,y,s)\in U_{T}\times U_{T}: \exists (p,q)\in B_{r}\times B_{r}: \Phi(\cdot, \cdot, \cdot, \cdot, p,q)\right.\\
&\text{achieves its monotone maximum at}\, (x,t,y,s), \\
&\text{for appropriate}\, \left.(h,k)\in \RR^{2}\right\}\subseteq U_{T}(\rho)\times U_{T}(\rho).
\end{align*}

By Proposition \ref{qviscosity}, this implies that 
\begin{align*}
|V|\geq C(1+\delta^{-1})^{-2d-2}r^{2d+2}&\geq C(1+A^{-(4-\al)/\al})^{-2d-2}A^{2d+2}\\
&\geq CA^{(8d+8)/\al}.
\end{align*}

If we denote the projection $\pi: \RR^{d+1}\times \RR^{d+1}\rightarrow \RR^{d+1}$ by $\pi((A,B))=A$, we have that 
\begin{equation}\label{touch_est}
\pi(V)\geq |U_{T}|^{-1}|V|\geq |Q_{1}|^{-1}|V|\geq CA^{(8d+8)/\al}.
\end{equation}

Finally, we note that for every $((x,t),(y,s))\in V$, since $\Phi(x,t,y,s,p,q)\geq 0$ for some $p,q\in B_{r}\subseteq B_{1}$, $\al\leq \frac{1}{2}$, and $A\leq 1$, this implies that 
\begin{equation}\label{4something}
|x-y|^{2}+|t-s|^{2}\leq C\delta\leq CA^{(4-\al)/\al}\leq CA^{6}.
\end{equation}

Next, we use \pref{touch_est} to show that there are points in $\pi(V)$, where $u$ can be approximated by a quadratic expansion. Let $\Sigma_{\ka}$ as in the $W^{3, \al}$ estimate (Theorem \ref{wpar}). 

By the $W^{3, \al}$ estimate, assuming that $U_{T}\subseteq Q_{1}$, 
\begin{equation}\label{west}
|U_{T}\setminus \Sigma_{\ka}(U_{T})|\leq |Q_{1}\setminus \Sigma_{\ka}(U_{T})\cap Q_{1/2}(0, -1/4)|\leq  C\ka^{-\al}.
\end{equation}
Although a priori, the two $\al$'s in \pref{touch_est} and \pref{west} are not necessarily the same, we can assume without loss of generality they are the same by taking the minimum of the two. 

Thus, if we let $\ka\geq CA^{-4(d+2)/\al^{2}}$, then 
\begin{equation*}
|U_{T}\setminus \Sigma_{\ka}(U_{T})|< |\pi(V)|,
\end{equation*}
which implies that $\pi(V)\cap \Sigma_{\ka}\neq \emptyset$. This implies that there are points  of $\pi(V)$ where $u$ can be touched monotonically in time by a quadratic expansion with controllable error, and the function $\Phi$ achieves it monotone maximum there. 

Finally, we show that there exists $M^{*}, y^{*}, s^{*}, G^{*}$ which satisfy the conclusion of the proposition. By the previous step, there exists $(x_{1}, t_{1}, y_{1}, s_{1})\in V$ with $(x_{1}, t_{1})\in \Sigma_{\ka}$. In other words, there exists $p,q\in B_{r}$, such that 
\begin{equation*}
\Phi(x_{1}, t_{1}, y_{1}, s_{1}, p, q)=\sup_{U_{T}(\rho)_\times U_{T}(\rho), \tau\leq t_{1}, \sigma\leq s_{1}} \Phi(x, \tau, y, \sigma, p,q),
\end{equation*}
and $(M, \xi, b)$ so that $|M|\leq \ka$, and for all $(x,t)\in U_{T}$, $t\leq t_{1}$,
\begin{align*}
\left| u(x,t)-u(x_{1}, t_{1})-b(t-t_{1})-\xi\cdot (x-x_{1})\right.&\left.-\frac{1}{2}(x-x_{1})\cdot M(x-x_{1})\right|\\
&\leq \frac{1}{6}\ka\left(|x-x_{1}|^{3}+|t-t_{1}|^{3/2}\right).
\end{align*}
Notice that since $u_{t}+\overline{F}(D^{2}u)=f(x,t)$ in $U_{T}$, and $u$ is touched from above and below at $(x_{1}, t_{1})$ by polynomials with Hessians equal to $M$, this implies that $b+\overline{F}(M)=f(x_{1}, t_{1})$. Therefore, defining
\begin{align*}
\phi(x,t):=&u(x_{1}, t_{1})+b(t-t_{1})+(\xi-p)\cdot (x-x_{1})+\frac{1}{2}(x-x_{1})\cdot M(x-x_{1})\\
&-\frac{1}{6}\ka \left(|x-x_{1}|^{3}+|t-t_{1}|^{3/2}\right),
\end{align*}
we have 
\begin{align}\label{omg1}
&u(x_{1}, t_{1})-v(y_{1}, s_{1})-\frac{1}{2\delta}\left[|x_{1}-y_{1}|^{2}+(t_{1}-s_{1})^{2}\right]\\
&\geq \sup_{U_{T}\times U_{T}, t\leq t_{1}, s\leq s_{1}} \left\{\phi(x,t)-v(y,s)-\frac{1}{2\delta}\left[|x-y|^{2}+(t-s)^{2}\right]-q\cdot (y-y_{1})\right\}.\notag
\end{align}

To control the right hand side from below, we consider that for any $(y,s)\in U_{T}$, with $s\leq s_{1}$, letting $x=x_{1}+y-y_{1}$, $t=t_{1}+s-s_{1}\leq t_{1}$, 
\begin{align}\label{omg2}
\sup_{(x,t)\in U_{T}, t\leq t_{1}} &\left\{\phi(x,t)-\frac{1}{2\delta}\left[|x-y|^{2}+(t-s)^{2}\right]\right\}\\
&\geq \phi(x_{1}+y-y_{1}, t_{1}+s-s_{1})-\frac{1}{2\delta}\left[|x_{1}-y_{1}|^{2}+(t_{1}-s_{1})^{2}\right]\notag\\
&=u(x_{1}, t_{1})+b(s-s_{1})+(\xi-p)\cdot (y-y_{1})+\frac{1}{2}(y-y_{1})\cdot M(y-y_{1})\notag\\
&-\frac{1}{6}\ka \left(|y-y_{1}|^{3}+|s-s_{1}|^{3/2}\right)-\frac{1}{2\delta}\left[|x_{1}-y_{1}|^{2}+(t_{1}-s_{1})^{2}\right].\notag
\end{align}

Combining \pref{omg1} and \pref{omg2} yields that 
\begin{align*}\label{omg3}
&u(x_{1}, t_{1})-v(y_{1}, s_{1})-\frac{1}{2\delta}\left[|x_{1}-y_{1}|^{2}+(t_{1}-s_{1})^{2}\right]\\
&\geq \sup_{(y,s)\in U_{T}, s\leq s_{1}}\left\{u(x_{1}, t_{1})+b(s-s_{1})+(\xi-p)\cdot (y-y_{1})+\frac{1}{2}(y-y_{1})\cdot M(y-y_{1})\right.\notag\\
&\left.-\frac{1}{6}\ka \left(|y-y_{1}|^{3}+|s-s_{1}|^{3/2}\right)-\frac{1}{2\delta}\left[|x_{1}-y_{1}|^{2}+(t_{1}-s_{1})^{2}\right]-v(y,s)\right.\\
&\left.-q\cdot (y-y_{1})\right\}.\notag
\end{align*}
This implies that 
\begin{align}
v(y_{1}, s_{1})&\leq \inf_{(y,s)\in U_{T}, s\leq s_{1}}\left\{v(y,s)-b(s-s_{1})-(\xi-p-q)\cdot (y-y_{1})\right.\\
&\left.-\frac{1}{2}(y-y_{1})\cdot M(y-y_{1})+\frac{1}{6}\ka \left(|y-y_{1}|^{3}+|s-s_{1}|^{3/2}\right)\right\}.\notag
\end{align}

Since $l\leq \eta$, let $M^{*}\in \mathbb{S}^{d}$ so that $M\leq M^{*}\leq M+C\eta^{\sigma}Id$, and $l^{-1}M^{*}$ has integer entries. Using that $\overline{F}$ is uniformly elliptic, $\overline{F}(M^{*})\leq \overline{F}(M)=f(x_{1}, t_{1})-b$. Let 
\begin{align*}
\Theta(y,s):=&v(y,s)-b(s-s_{1})-(\xi-p-q)\cdot (y-y_{1})\\
&-\frac{1}{2}(y-y_{1})\cdot (M-C\eta^{\sigma}Id)(y-y_{1})+\frac{1}{6}\ka \left(|y-y_{1}|^{3}+|s-s_{1}|^{3/2}\right).
\end{align*}

By \pref{lilroom},
\begin{align*}
\Theta_{s}&+F(M^{*}+D^{2}\Theta, y, s)= v_{s}-b+\frac{1}{4}\ka|s-s_{1}|^{1/2}\\
&+F\left(M^{*}+D^{2}v-M+C\eta^{\sigma}Id+\frac{1}{2}\ka |y-y_{1}|Id+\frac{1}{2}\ka \frac{(y-y_{1})\otimes (y-y_{1})}{|y-y_{1}|}, y, s\right)\\
&\geq v_{s}-b+F(D^{2}v, y, s)-C\left(M^{*}-M+C\eta^{\sigma}Id+C\frac{1}{2}\ka |y-y_{1}|Id\right)\\
&\geq f(y,s)+cA-b-C\eta^{\sigma}-C\frac{1}{2}{\ka}|y-y_{1}|\\
&\geq f(y,s)+cA-b-C\eta^{\sigma}-C\frac{1}{2}(\ka+1)|y-y_{1}|\\
&\geq \overline{F}(M)-CA^{6}+cA-C\eta^{\sigma}-C\frac{1}{2}(\ka+1)|y-y_{1}|,
\end{align*}
where the last line holds by \pref{4something}, and using that $\overline{F}(M)=f(x_{1}, t_{1})-b$. 

This implies that in $Q_{cA(\ka+1)^{-1}}(y_{1}, s_{1})$, 
\begin{equation*}
\Theta_{s}+F(M^{*}+D^{2}\Theta, y, s)\geq \overline{F}(M)-CA^{6}+cA-C\eta^{\sigma}.
\end{equation*}

In addition, comparing \pref{omg3} and the definition of $\Theta$, 
\begin{equation}\label{useful}
\Theta(y_{1}, s_{1})\leq \inf_{(y,s)\in U_{T}, s\leq s_{1}} \left(\Theta-C\eta^{\sigma}|y-y_{1}|^{2}\right).
\end{equation}
Let $(y^{*}, s^{*})$ so that $(\eta^{-1}y^{*}, \eta^{-2}s^{*})\in \ZZ^{d+1}$, and $d[(y^{*}, s^{*}), (y_{1}, s_{1})]\leq \sqrt{d}\eta$. 

Let 
\begin{equation*}
G^{*}:=(y^{*}, s^{*})-\eta G_{0}.
\end{equation*}
Since $(y_{1}, s_{1})\in U_{T}(\rho)$, $d[(y^{*}, s^{*}), \partial_{p}U_{T}]\geq \rho-\sqrt{d}\eta\geq \sqrt{d}\eta$ so long as $\rho:=CA^{1/\al}\geq C\eta$ (which is satisfied if $\sigma\leq \al$). This implies that $G^{*}\subseteq U_{T}$. 

We next claim that $G^{*}\subseteq Q_{cA(\ka+1)^{-1}}(y_{1}, s_{1})$ for an appropriate choice of $\ka$. Let $\ka:=\eta^{\sigma-1}$ with $\sigma:=((1+4(d+2))/\al^{2})^{-1}\leq \al$. Since $A\geq C\eta^{\sigma}$, we may choose the constants so that $cA(\ka+1)^{-1}\geq \sqrt{d}\eta$.  This yields that $G^{*}\subseteq Q_{cA(\ka+1)^{-1}}(y_{1}, s_{1})$, as asserted. 

Therefore, 
\begin{equation}
\Theta_{s}+F(M^{*}+D^{2}\Theta, y, s)\geq \overline{F}(M^{*})\quad\text{in}\quad G^{*}.
\end{equation}

By \pref{useful}, we conclude that 
\begin{equation}\label{useful2}
\inf_{G^{*}} \Theta\leq \inf_{\partial_{p}G^{*}}\Theta -C\eta^{\sigma}.
\end{equation}
This implies by Lemma \ref{muabp} and \pref{useful2}, that 
\begin{equation*}
\mu(G^{*}, \overline{F}(M^{*}), M^{*})\geq cA^{d+1},
\end{equation*}
and this completes the proof.
\end{proof}

Finally, we are ready to prove Theorem \ref{mainthm}.
\begin{proof}[Proof of Theorem \ref{mainthm}]
{
We prove a rate in probability for the decay of $u-u^{\ve}$. Fix $M_{0}$ and $U_{T}$ so that $U_{T}\subset Q_{1}$, and 
\begin{equation*}
\left(1+K_{0}+\norm{g}_{C^{0,1}(\partial_{p}U_{T})}\right)\leq M_{0}. 
\end{equation*}
We will show that there exists a $\beta>0$ and a random variable $\mathcal{X}:\Om\rightarrow \RR$ such that 
\begin{equation*}
\sup_{U_{T}} \left\{u(x,t)-u^{\ve}(x,t,\om)\right\}\leq C\left[1+\ve^{p}\mathcal{X}(\om)\right]\ve^{\beta}.
\end{equation*}}
{
We mention that a rate on $u^{\ve}-u$ follows by a completely analogous argument for $\mu^{*}$, so we choose to omit it. }

{
Fix $\ve\in (0,1)$, fix $p<d+2$, and let $\sigma$ as in Proposition \ref{grid}. Let $\al$ be the $\al$ associated with $p$ as in Corollary \ref{cormup}, and let $q:=\frac{p}{4}$. Choose $m$ to so that 
\begin{equation}\label{mdef}
\max\left\{3^{-m/4}, 3^{-m\al/(d+1)}\right\}\leq \ve.
\end{equation}}

{
In the language of Proposition \ref{grid}, let $\eta:=3^{-m\al/2(d+1)}$, and choose $l:=3^{-m\al/2d}$. Notice that we have that $l\leq \eta\leq \ve^{1/2}$. This implies that for any $A\geq C\eta^{\sigma}$,
\begin{align*}
&\left\{\om: \sup_{(x,t)\in U_{T}} u(x,t)-u^{\ve}(x, t, \om)\geq A\right\}\\
&\subseteq \bigcup_{(y, s, M)\in \mathcal{I}(A)}\left\{\om: \mu\left((y/\ve, s/\ve^{2})+\eta \ve^{-1}G_{0}, \om, \overline{F}(M), M\right)\geq cA^{d+1}\right\}\\
&= \bigcup_{(y, s, M)\in \mathcal{I}(A)}\left\{\om: \mu\left((y/\ve, s/\ve^{2})+G_{m},\om, \overline{F}(M), M\right)\geq cA^{d+1}\right\}
\end{align*}
where 
\begin{equation*}
\mathcal{I}(A):=\left\{(y, s, M): (y,s)\in Q_{1}, (\eta^{-1}y, \eta^{-2}s)\in \ZZ^{d+1}, |M|\leq 3^{m\al/2(d+1)}\right\}.
\end{equation*}
This is possible since $\eta<1$ and  Proposition \ref{grid} yields that $\sigma<1$, which implies that  $|M|\leq \eta^{\sigma-1}\leq \eta^{-1}\leq 3^{m\al/2(d+1)}$. We mention also that $l^{-1}M\in \mathbb{Z}^{d^{2}}\cap \mathbb{S}^{d}$. }

This implies that 
{
\begin{equation}
\sup_{(x,t)\in U_{T}} \left\{u(x,t)-u^{\ve}(x,t, \om)\right\}\leq cA^{d+1}+\mathcal{Y}_{m}(\om)
\end{equation}
where 
\begin{equation}
\mathcal{Y}_{m}(\om):=\left\{\sup \mu((z, r)+G_{m}, \om, \overline{F}(M), M): (z\ve^{-1},r\ve^{-2}, M)\in \mathcal{I}(A)\right\}. 
\end{equation}}
{
To find the number of elements in $\mathcal{I}(A)$, consider that since $\eta^{-1}z\in \ZZ^{d}\cap Q_{1/\ve}$ and $\eta^{-2}s\in \ZZ\cap [0, 1/\ve^{2}]$, there are $(\ve \eta)^{-(d+2)}$ choices for $(z,s)$. This implies that there are at most $3^{3m\al}$ choices. For the matrices, consider that since $3^{m\al/2d}M\in \mathbb{Z}^{d^{2}}\cap \mathbb{S}^{d}$ and $|M|\leq 3^{m\al/2(d+1)}$, this implies that there are at most $3^{m\al(d+1)}$ terms. Total, there are $3^{m\al(d+4)}$ combinations to choose from in $\mathcal{I}(A)$.} 

By Corollary \ref{cormup}, for each $(z, r, M)\in \mathcal{I}(A)$, 
{
\begin{equation*}
\PP[(z,r)+\mu(G_{m}, \om, \overline{F}(M), M)\geq (1+|M|)^{d+1}3^{-m\al}\tau]\leq C\exp(-c3^{mp}\tau).
\end{equation*}

Since $|M|^{d+1}\leq 3^{m\al/2}$, this implies that 
\begin{equation*}
\PP[(z,r)+\mu(G_{m}, \om, \overline{F}(M), M)\geq 3^{-m\al/2}\tau]\leq \exp(-c3^{mp}\tau).
\end{equation*}

Using a union bound and summing over all of the terms in $\mathcal{I}(A)$, 
\begin{equation*}
\PP[\mathcal{Y}_{m}(\om)\geq 3^{-m\al/2}\tau]\leq C3^{m\al(d+4)} \exp(-c3^{mp}\tau)\leq C\exp(-c3^{mp}\tau).
\end{equation*}}

Replacing $\tau$ by $\tau+1$, we have that for all $\tau\geq 0$, 
\begin{equation*}
\PP[(3^{m\al/2}\mathcal{Y}_{m}(\om)-1)_{+}\geq \tau]\leq C\exp(-c3^{mp}\tau).
\end{equation*}

Replacing again $\tau\rightarrow 3^{-mq}\tau$ yields that 
\begin{equation*}
\PP[3^{mq}3^{m\al/2}\left(\mathcal{Y}_{m}(\om)-1\right)_{+}\geq \tau]\leq C\exp(-c3^{m(p-q)}\tau).\end{equation*}

Summing over $m$, using that $p>q$, this implies that for all $\tau\geq 0$. 
\begin{align}\label{cdf}
\PP\left[\sup_{m} \left\{3^{mq}3^{m\al/2}\left(\mathcal{Y}_{m}(\om)-1\right)_{+}\right\}\geq \tau\right]&\leq \sum_{m}\PP[3^{mq}3^{m\al/2}\left(\mathcal{Y}_{m}(\om)-1\right)_{+}\geq \tau]\notag\\
&\leq C\exp(-c\tau).
\end{align}
{
Letting 
\begin{equation}
\mathcal{X}(\om):=\sup_{m}\left\{ 3^{mq}\left(3^{m\al/2}\mathcal{Y}_{m}(\om)-1\right)_{+}\right\}
\end{equation}
and integrating \pref{cdf} in $\tau$ yields that 
\begin{equation}\label{xexp}
\EE[\exp(\mathcal{X}(\om))]\leq C.
\end{equation}}

This implies that
{
\begin{align*}
\sup_{(x,t)\in U_{T}} \left\{u(x,t)-u^{\ve}(x,t,\om)\right\}&\leq C\eta^{\sigma(d+1)}+C(3^{-mq}\mathcal{X}(\om)+1)3^{-m\al/2}\\
&\leq C(1+\ve^{p}\mathcal{X}(\om))\ve^{\beta}
\end{align*}}
{
for some choice of $\beta$, where $\beta(\la, \La, d, p)$.}
\end{proof}

\section*{Acknowledgements}
{
Part of this article appeared in first author's doctoral thesis. Both authors would like to thank Scott Armstrong and Takis Souganidis for useful discussions. The first author was partially supported by NSF grants DGE-1144082 and DMS-1147523.  The second author was partially supported by NSF grant DMS-1461988 and the Sloan Foundation.  This collaboration took place at the Mittag-Leffler Institute. 
}


\bibliographystyle{amsplain}
\bibliography{paralg}


\end{document}